\documentclass[12pt]{iopart}

\usepackage[latin1]{inputenc}
\usepackage{iopams} 
\usepackage{amsthm}
\usepackage{lscape}
\usepackage{graphicx}
\usepackage{makeidx}
\usepackage{mathtools}
\usepackage[caption=false,font=footnotesize]{subfig}

\theoremstyle{plain}%
\newtheorem{thm}{Theorem}[section]
\newtheorem{lem}[thm]{Lemma}

\newtheorem*{cor}{Corollary}

\newtheorem{remark}{Remark}

\newcommand\norm[1]{\left\lVert#1\right\rVert}

\newcommand\expval[1]{\langle #1 \rangle}

\usepackage{xcolor}
\usepackage[ruled,vlined,boxed,commentsnumbered]{algorithm2e}

\usepackage[nomessages]{fp}%

\usepackage{scalerel,stackengine}
\stackMath
\newcommand\reallywidehat[1]{%
\savestack{\tmpbox}{\stretchto{%
  \scaleto{%
    \scalerel*[\widthof{\ensuremath{#1}}]{\kern.1pt\mathchar"0362\kern.1pt}%
    {\rule{0ex}{\textheight}}%
  }{\textheight}%
}{2.4ex}}%
\stackon[-6.9pt]{#1}{\tmpbox}%
}

\begin{document}

\title[Non-Stationary Multi-layered Gaussian Priors for Bayesian Inversion]{Non-Stationary Multi-layered Gaussian Priors for Bayesian Inversion}

\author{Muhammad Emzir$^{1)}$, Sari Lasanen$^{2)}$, Zenith Purisha$^{1)}$, Lassi Roininen$^{3)}$, Simo S\"arkk\"a$^{1)}$}
\address{1) Department of Electrical Engineering and Automation, Aalto University, P.O. Box 12200, FI-00076 Aalto, Finland}
\address{2) Sodankyl\"a Geophysical Observatory, University of Oulu, P.O. Box 8000, FI-90014 University of Oulu, Finland}
\address{3) School of Engineering Science, Lappeenranta-Lahti University of Technology, P.O. Box 20, FI-53851 Lappeenranta, Finland}

\ead{muhammad.emzir@aalto.fi}
\vspace{10pt}
\begin{indented}
\item[]March 2020
\end{indented}
\begin{abstract}
In this article, we study Bayesian inverse problems with multi-layered Gaussian priors. We first describe the conditionally Gaussian layers in terms of a system of stochastic partial differential equations. We build the computational inference method using a finite-dimensional Galerkin method. We show that the proposed approximation has a convergence-in-probability property to the solution of the original multi-layered model. We then carry out Bayesian inference using the preconditioned Crank--Nicolson algorithm which is modified to work with multi-layered Gaussian fields. We show via numerical experiments in signal deconvolution and computerized X-ray tomography problems that the proposed method can offer both smoothing and edge preservation at the same time.
\end{abstract}
\section{Introduction}
\label{sec:intro}

The Bayesian approach provides a consistent framework to obtain solutions of inverse problems. By formulating the unknown as a random variable, the degree of information that is available can be encoded as a statistical prior. The ill-posedness of the problem is mitigated by reformulating the inverse problem as a well-posed extension in the space of probability distributions \cite{Kaipio2004}. Among statistical priors that are commonly used in Bayesian inverse problem is the Gaussian prior which is relatively easy to manipulate, has a simple structure, and also has a close relation with traditional Tikhonov regularization. This approach has also get a growing interest from a machine learning community, where the use of Gaussian prior for Bayesian inference is known as Gaussian process regression \cite{CarlEdwardRasmussen2005}.

When the unknown is a multivariate function, it is natural to model it as a random field. There is a vast amount of studies on Gaussian random fields and their applications where the random field is assumed to be stationary \cite{Heaton2018}. Stationary Gaussian fields have uniform spatial behavior. As a result, stationary Gaussian fields fail in the cases where the smoothness of the target varies spatially in unexpected ways \cite{Paciorek2004,Fuglstad2015}.  To model variable spatial behaviour, the covariance structure needs to be tuned appropriately. There have been a number of proposals to increase flexibility of the non-stationary Gaussian fields. One of the earliest strategies is to construct an anisotropic variant of an isotropic covariance function \cite{Paciorek2003,Paciorek2004,snelson2004warped}. Another approach is to reformulate the fields as stochastic partial differential equations and let some of the coefficients vary in space \cite{Lindgren2011}. 
In \cite{Roininen2016}, a similar idea is used, where instead of a predetermined length-scale function, a random field is used. They explicitly choose the Gaussian fields to have Mat\`{e}rn covariance functions and varying their length-scale according to another Gaussian field. This approach is recently extended to allow some flexibility in measurement noise model and hyperprior parameters \cite{Monterrubio_G_mez_2020}. A different approach is used in \cite{Damianou2013} where the model is formed as a cascaded composition of Gaussian fields (see also \cite{Duvenaud2014}).

There are some recent findings analyzing how adding more layers translate to the ability of the overall hierarchical Gaussian field to describe random fields with complex structures. It has been demonstrated in \cite{Duvenaud2014} that as the number of layers increases, the density of the last Gaussian field shrinks to a one-dimensional manifold. This might prevent cascaded Gaussian fields to model phenomena where the underlying dimension is greater than one. Ergodicity and effective depth of a hierarchical Gaussian fields has also been analyzed in \cite{Dunlop2018}. The consequence of their result is that, there might be only little benefit in increasing number of layers after reaching certain number.
\begin{figure}[!h]
\centering
	\includegraphics[width=0.9\textwidth,trim={0cm 0cm 0cm 0cm},clip]{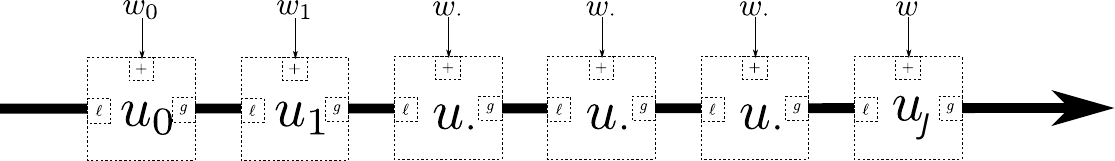} 
	\caption{Illustration of a chain of Gaussian fields. Each node (field) has an independent white noise input field and a length-scale parameter $\ell$. The length-scale $\ell$ is obtained as a function  evaluated on the fields above it. These Gaussian fields are approximated in finite-dimensional Hilbert space using $\mathbf{L}(\mathbf{u}_{j-1})\mathbf{u}_j = \mathbf{w}_j$, see \eref{eqn:prior}.
	}\label{fig:Lu_relations}
\end{figure}

In many spatio-temporal inverse problems, it is often useful to start by  working in an infinite-dimensional space. There are essentially two approaches to construct a Bayesian inference algorithm. The first approach is to discretize the forward map (e.g., via grid partitions or finite element meshes) and apply a Bayesian inference method  to the finite-dimensional setting \cite{Kaipio2004}. The second approach \cite{Stuart2010,Dashti2017} is to directly apply Bayesian inference to the infinite-dimensional problem, and afterward, apply a discretization method. The latter approach is possible through realizing that the posterior and prior probability distribution can be related via the Radon--Nikodym derivative which can be generalized to function spaces \cite{Stuart2010}.

For any of the aforementioned approaches, the sampling technique has to be carefully designed. The traditional Markov chain Monte Carlo (MCMC) algorithms suffer from slow mixing times upon the grid partitioning refinement \cite{Cotter2013}. These methods dictate to reduce the MCMC step size which becomes computationally expensive. There are several MCMC algorithms designed specifically to deal with infinite-dimensional problems so that mixing time will, up to some extent, be almost independent of the dimensionality \cite{Cotter2013,Law2014,Beskos2017,Rudolf2018}. %
Among these algorithms, the preconditioned Crank-Nicolson (pCN) \cite{Cotter2013} is a very simple one to be implemented. The main idea of this algorithm is to design a random walk such that the discretization of this random walk is invariant for the target measure which can be designed to be the posterior measure. Assuming that the target measure has a density with respect to a Gaussian reference measure, the pCN algorithm takes advantage of a clever selection of the Markov transition kernel. It has been shown in \cite{Hairer2014} via spectral gap analysis that pCN has dimension-independent sampling efficiency. This benefit comes in contrast to the standard random walk proposal where the probability of the proposal acceptance will be almost zero in infinite-dimensional case. Recently, a non-centered version of this algorithm was introduced in \cite{Chen2018}. Their work was developed using a non-centered reparametrization developed in \cite{Papaspiliopoulos2007}. 
This transformation is important in the hierarchical prior Bayesian inversion cases since it breaks the dependency between parameters in different levels which simplifies the calculation of posterior. There are also a generalization of the pCN algorithm to take into account the information of the measure and an adaptive version of it \cite{Rudolf2018,Hu_2017}.

The main contribution of this article is to present a method for Bayesian inverse problems with multi-layered Gaussian field prior models via a Galerkin method. The motivation is to have enough complexity in the model to allow for both smoothing and edge preserving properties while keeping relatively low number of layers at the same time. In particular, we follow the approach recently described in \cite{Roininen2016}, where the Gaussian fields are represented as stochastic partial differential equations (SPDEs), in which the length-scale parameters depend on the solution of SPDEs for the layer above. However, instead of using a grid partition, we will apply Galerkin method, on which we have already obtained preliminary results in \cite{Emzir2019}. %
Using this approach, we can avoid evaluating SPDE forward problem via finite difference equations, and the number of parameters to be evaluated is greatly reduced. This can be considered as a compromise between the accuracy and the computational complexity. Our approach is also related to the one described in \cite{Solin2019}, where Gaussian fields with stationary covariance functions are approximated in finite-dimensional Hilbert spaces. Such approach relies on the fact that the covariance function of a stationary Gaussian field is expandable via Mercer's theorem. %
In this work, we construct several Gaussian fields in a hierarchical structure. Using this structure, we show how to transform this model into a chain of Gaussian fields driven by white noise fields. This formulation enable us to use the Bayesian inference framework described in \cite{Dunlop2018}. We then implement the proposed approximation into an MCMC algorithm. 
Our MCMC algorithm is based on the non-centered version of the preconditioned Crank--Nicolson algorithm modified to work with multi-layered Gaussian fields \cite{Dunlop2018,Chen2018}.

Since translating the inverse problem into finite-dimensional formulation will introduce a discretization problem that could exacerbate the reconstruction errors \cite{Lassas2004,Kaipio2007}, we rigorously show that the proposed approximation enjoys a nice convergence-in-probability property to the weak solution of the forward model. To show the convergence result, we start by establishing an upper bound for the square root of the precision operator of the Gaussian field. From here, we develop another upper bound for the error between this operator and its finite-dimensional approximation. We then show that the approximate solution satisfies a H\"older continuity property. Using these results and additional tightness conditions, we finally show that the proposed approximation converges in probability to the original weak solution of the forward model. As a consequence, we can guarantee that the approximated prior and posterior probability distribution converges weakly to the original prior and posterior, respectively.

In this article, we also present an application of the proposed method to a computerized X-ray tomography problem. Computerized tomography problems are very challenging since they are ill-posed \cite{Natterer2001}. One way is to compute maximum a posteriori estimate for a general X-ray tomography problem using a Gaussian prior in finite-dimensional setting \cite{Tarantola2005}. In \cite{Li2013}, a Bayesian method using Gaussian fields with non-stationary covariance functions described in \cite{Paciorek2004,Plagemann2008} is developed for plasma fusion and soft X-ray tomography. Recently in \cite{Purisha2019}, a Hilbert space approximation technique described in \cite{Solin2019} is used in a sparse tomographic inverse problem. Bayesian tomographic reconstruction with non-Gaussian prior has been studied in \cite{Frese2002,Chen2011}. There are also some recent results in X-ray tomography using  deep learning methods. However, contrary to the Bayesian and regularization approaches, these methods are prone to instabilities when exposed to a small perturbation and structural changes \cite{Antun_2020}.

Previously, we have presented a subset of our contributions in \cite{Emzir2019}. In the present work, we extend the methods presented in \cite{Emzir2019} to Bayesian inverse problems and we have also added a throughout convergence analysis of the methods. The algorithm presented in this work also generalizes the algorithm presented in \cite{Emzir2019} to the case of multiple hyperprior layers. 

The outline of the article is the following. In Section \ref{sec:Hilbert}, we present a formulation of Bayesian inverse problems using multi-layered Gaussian priors via Galerkin method. The convergence analysis is presented in Section~\ref{sec:Convergence}.
In Section \ref{sec:Sampling}, we propose a Markov chain Monte Carlo algorithm to sample the Fourier coefficients from their posterior distribution. In Section \ref{sec:Application}, we present an application of the proposed method to a one-dimensional example model and to a tomographic inverse problem. Finally, Section ~\ref{sec:conclusions} concludes the article.

\subsection{Notation}
Let $\left\{\phi_l \right\}$ be a basis formed from the orthogonal eigenfunctions of the Laplace operator with respect to some domain $\Omega$ with suitable boundary conditions. Let $N$ be the number of basis functions $\left\{\phi_l \right\}$ used in the Galerkin method and let $H_N$ denote their span. For multi-layer Gaussian field priors, we use $u_j$ to denote the random field at layer $j$, where the total number of layer is $J+1$, that is, $J$ is the number of hyperprior layers. The collection of $J+1$ random fields $\left(u_0,u_1,\ldots,u_{J} \right)$ is denoted by $u$. The Fourier transform of any random field $z$ is denoted by $\widehat{z}$, where the Fourier coefficient for index $l$ is given by $\widehat{z}(l)$, that is, $\widehat{z}(l) = \langle z, \phi_l\rangle$, where $\langle \cdot , \cdot \rangle$ is the standard $L^2$ inner product on the domain $\Omega$. We use bold characters to denote vectors or matrices with elements in $\mathbb{C}$ or $\mathbb{R}$. Fourier coefficient for random field $u_j$ for index $-N$ to $N$ is given by $\mathbf{u}_j$, that is, $\mathbf{u}_j = \left(\widehat{u}_j(-N) \dots \widehat{u}_j(N)\right)$. We denote $J+1$ collections of the Fourier coefficients of random fields $u$ as $\mathbf{u} = \left( \mathbf{u}_0,\ldots,\mathbf{u}_J \right)$.
The identity operator is denoted with $I$.

\section{Finite-dimensional approximations}

\label{sec:Hilbert}
Consider a Bayesian inverse problem on a Gaussian field where the unknown is a real-valued random field $\upsilon(\mathbf{x}): \Omega \rightarrow \mathbb{R}$ on a  bounded domain $\Omega \subset \mathbb{R}^d$. We assume in the inverse problem that $\upsilon$ belongs to a Hilbert space $H$, specifically, $\upsilon \in H := L^2(\Omega)$. To carry out the Bayesian inference, a set of measurements is taken (either direct or indirect in multiple locations). In this article, the measurement is assumed to be a linear operation on $\upsilon$ corrupted with additive noises, that is, $y_k = \expval{\upsilon,h_k}  + e_k$, where $h_k$ is an element in $H$ represents a real linear functional on $H$, and $e_k$ is a zero-mean white noise with a covariance matrix $\mathbf E$.

In the Bayesian framework, the estimation problem is equivalent to exploring the posterior distribution of $\upsilon$ given the measurements $\{y_k\}$. The Bayesian inversion approach for this problem starts with assuming that $\upsilon$ is a Gaussian field with a certain mean (assumed zero for simplicity) and covariance function $C(\mathbf{x},\mathbf{x}')$.
In the case when $C(\mathbf{x},\mathbf{x}')$ is a Mat\'ern covariance function, the Gaussian field $\upsilon$ can be generated from a stochastic partial differential equation of the form \cite{Lindgren2011,Roininen2016}
\begin{equation}
\left(1-\ell^2 \Delta \right)^{\alpha/2} \, \upsilon(\mathbf{x}) = \sqrt{\beta \ell ^d} \, w(\mathbf{x}), \label{eq:SPDE_Stationary}
\end{equation}
where $\alpha = \nu+d/2$, $d$ is the dimension of the space, $\nu$ is a smoothness parameter, $w(\mathbf{x})$ is a white noise on $\mathbb{R}^d$, $\ell$ is the length-scale constant of the Mat\'ern covariance function $C$, and $\beta = \sigma^2 2^d \pi^{d/2} \Gamma(\alpha)/\Gamma(\nu)$ with $\sigma^2$ being a scale parameter.

To obtain a non-stationary field, we modify SPDE \eref{eq:SPDE_Stationary} so that the length-scale $\ell$ is modeled via another Gaussian field $u$ with Mat\`{e}rn covariance function. Namely, we select $\ell(\mathbf{x}) = g(u(\mathbf{x}))$, where $g$ is a smooth  positive function $g : \mathbb{R} \rightarrow \mathbb{R}_+$. As in \cite{Roininen2016}, we also require that $\ell$ should satisfy $\sup_{\mathbf{x}\in \Omega} \ell(\mathbf{x}) < \infty$ and $ \inf_{\mathbf{x}\in \Omega} \ell(\mathbf{x})>0$ with probability one. 
For notational and mathematical convenience we restrict $\alpha=2$. The results of this article could also be extended to other cases.
Introducing the spatially varying length-scale $\ell(\mathbf{x})=g(u(\mathbf{x}))$ into \eref{eq:SPDE_Stationary}, and since the length-scale is always greater than zero, with probability one, using $\kappa = 1/\ell$ we obtain the following SPDE %

\begin{equation}
	\left(\kappa(u(\mathbf{x}))^{2} - \Delta \right) \, \upsilon(\mathbf{x}) =  \, \sqrt{\beta} \kappa(u(\mathbf{x}))^{\nu} w(\mathbf{x}). \label{eq:SPDE_Nstationary_kappa}
\end{equation}

In order to facilitate easy operation with the Laplace operator, we choose to expand $v$ using the eigenfunctions of the Laplacian. With a suitable boundary condition, the Laplace operator can be expressed $-\Delta \upsilon = \sum_{j=-\infty}^{\infty} \lambda_j \expval{\upsilon,\phi_j} \phi_j$, 
where ${\phi_j}$ is a complete set of orthonormal eigenfunctions of $\Delta$
and $\lambda_j>0$, where $\lim\limits_{j \rightarrow \infty} \lambda_j = \infty$ \cite{Evans2014}. Observe that, in the sense of \eqref{eq:SPDE_Nstationary_kappa}, $\kappa(u(\mathbf{x})) := 1/g(u(\mathbf{x})) \in L^\infty(\Omega)$ is a multiplication operator acting pointwise, that is, $(\kappa(u) \upsilon)(\mathbf{x}) = \kappa(u(\mathbf{x}))\upsilon(\mathbf{x}), \forall \mathbf{x} \in \Omega$ \cite{Lasanen2018}.  %

In what follows, we will first describe the matrix representation of a chain of Gaussian fields generated from SPDEs in the form of \eref{eq:SPDE_Nstationary_kappa} for $d$-dimensional domain. Then in Section \ref{sec:Convergence}, we will develop a convergence result of the Galerkin method developed here to the weak solution of \eref{eq:SPDE_Nstationary_kappa}.

\subsection{Matrix representation}

Let us examine a periodic boundary condition on $d$-dimensional box $\Omega$ with side length $1$. Within this boundary condition, it is useful to consider $H$ as a complex Hilbert space, so that we can set $\phi_l = \exp(i\;c_d\;\mathbf{x}^\top \mathbf{k}(l))$ as Fourier complex basis functions for $d$ dimensions, for some constant $c_d$ and a multi-index $\mathbf{k}(l)$ which is unique for every $l$. Let the finite-dimensional Hilbert subspace $H_N$ of $H$ be the span of ${\phi_{-N},\ldots,\phi_0,\ldots,\phi_{N}}$.
In what follows, we will explicitly construct the multi-index $\mathbf{k}(\cdot)$.

Without losing generality, let us assume that every entry in $\mathbf{k}(l)$ is between $-n$ to $n$. 
For any $-N \leq l\leq N$, where $N = \frac{(2n+1)^d-1}{2}$, we would like to construct $\mathbf{k}(l) =  (k_1(l), \ldots,k_d(l))$ such that it is unique for each $-N\leq l \leq N$ and $\mathbf{k}(l+m) = \mathbf{k}(l)+\mathbf{k}(m)$, $-N\leq l,m \leq N$ and $\max(|{k_r(l+m)}|) < n, r\leq d$. %
The construction of $\mathbf{k}: [-N,N] \rightarrow [-n,n]^d$ is as follows. Let the matrix $\mathbf{K}(n)$ be given as
\begin{align*}
	\mathbf{K}(n) = 
   \begin{pmatrix}
   \mathbf{z}_n \otimes \mathbf{e}_n^{\otimes d-1} \\  \mathbf{e}_n \otimes z_n \otimes \mathbf{e}_n^{\otimes d-2} \\ \vdots \\ \mathbf{e}_n^{\otimes d-1} \otimes \mathbf{z}_n
   \end{pmatrix},
\end{align*}
where $\mathbf{z}_n = (-n,-n+1,\ldots,n-1,n)$, $\mathbf{e}_n = (1,1,\ldots,1,1)$, and $\mathbf{e}_n^{\otimes d}$ is a Kronecker product of $\mathbf{e}_n$ repeated for $d$ times.
The multi index $\mathbf{k}(l)$ is given by selecting $l+(N+1)$-th column of $\mathbf{K}$. The linear relation is defined only if $-N\leq l+m \leq N$ and every element of the summation $\mathbf{k}(l)+\mathbf{k}(m)$ has values in $[-n,n]$. As an example let $d=2$ and $n=1$. This gives us:
\begin{align*}
	\mathbf{K}(1) = 
	\begin{pmatrix}
	-1&-1&-1&0&0&0&1&1&1\\
	-1&0&1&-1&0&1&-1&0&1
	\end{pmatrix}.	
\end{align*}

It can be verified that $\mathbf{k}(l)$ selected this way is both unique and linear, given that the summation result is inside the range. For example, $\mathbf{k}(1)+\mathbf{k}(2) = \mathbf{k}(3)$. However, $\mathbf{k}(1)+\mathbf{k}(1)$ is not defined since the summation is outside the limit. In the following, if the context is clear, we will also use the multiple index $\mathbf{k}(l)$ for Fourier component of $u$, that is $\widehat{u}(\mathbf{k}(l)) := \widehat{u}(l)$.

Since we have no information outside of the frequencies of interest, under the periodic boundary condition, for $-N\leq l,m \leq N$, $\expval{\phi_m u \;,\phi_l} = \widehat{u}(l-m)$, when $\max(|\mathbf{k}(l) - \mathbf{k}(m)|) \leq n$, and zero elsewhere.

Let us denote with $\mathbf{u}$, $\mathbf{v}$, and $\mathbf{w}$ finite-dimensional representations of $u$, $\upsilon$, and $w$, respectively. Let $M_N(u)$ be the matrix representation of the multiplication operator $u(\mathbf{x})$ on $H_N$. For a random field $r$ with Fourier coefficients $\mathbf{r} = (\widehat{r}(-m) \cdots \widehat{r}(m))^\top \in \mathbb{C}^{2m+1}$, let us write a Toeplitz matrix $T\in \mathbb{C}^{(m+1)\times (m+1)}$ with elements from $\mathbf{r}$ as follows:
\begin{align*}
	T(\mathbf{r}) = 
	\begin{pmatrix}
		\widehat{r}(0)& \cdots & \widehat{r}(-m)\\
		\vdots & \ddots & \vdots\\
		\widehat{r}(m)& \cdots & \widehat{r}(0)
	\end{pmatrix}.
\end{align*}
The previous discussion allows us to write for $d=1$, $M_N(u)  := T(\tilde{\mathbf{u}}) \in \mathbb{C}^{(2N+1) \times (2N+1)}$, 
where $\tilde{\mathbf{u}} = (\mathbf{0}_{1\times N} , \mathbf{u}^\top , \mathbf{0}_{1\times N})^\top \in \mathbb{C}^{(4N+1) \times 1}$. It is also possible to construct $M_N(u)$ for $d>1$. However, instead of working directly on $\mathbf{u}$, we need to work on the frequency indices. Let $\mathbf{J} \in \mathbb{R}^{(2n+1)\times (2n+1)}$ be a square matrix where all of its entry equal to one. Let also $\mathbf{Z}^{(1)} = T(\mathbf{z}_{2n}) \otimes \mathbf{J}^{\otimes d-1},\mathbf{Z}^{(2)} = \mathbf{J} \otimes T(\mathbf{z}_{2n}) \otimes \mathbf{J}^{\otimes d-2},\ldots, \mathbf{Z}^{(d)} = \mathbf{J}^{\otimes d-1} \otimes T(\mathbf{z}_{2n})$, respectively. 
Using these matrices, the $(l,m)$-th entry of $M_N(u)$ is given by 
\begin{align}
 M_N(u)_{l,m} = \widehat{u}(\tilde{\mathbf{k}}(l,m)),  \label{eq:Mn_d_more_than_1} 
\end{align}
where $\tilde{\mathbf{k}}(l,m) = \left(\mathbf{Z}^{(1)}_{l,m},\cdots,\mathbf{Z}^{(d)}_{l,m}\right)$. In this equation, we assign $\widehat{u}(\tilde{\mathbf{k}}(l,m)) = 0 $ when $\max(|\tilde{\mathbf{k}}(l,m)|) > n$.

The sparsity of $M_N$ as $n$ approaches infinity is $(\frac{3}{4})^d$. %
The weak solution to \eref{eq:SPDE_Nstationary_kappa} in the span of $H_N$ is equivalent to the following equation,
\begin{align}\label{eqn:prior}
\mathbf{L}(u)\mathbf{v} &=  \mathbf{w}, 
\end{align}
where $\mathbf{L}(u):= \frac{1}{\sqrt{\beta}}( M_N(\kappa(u)^{d/2})-M_N(\kappa(u)^{-\nu}) \mathbf{D})$ is the square root of the precision operator corresponds to $\upsilon$ in matrix form, $\mathbf{D}$ is a diagonal matrix, and $\mathbf{v},\mathbf{w}$ are complex vectors with appropriate dimensions. The diagonal entries of $\mathbf{D}$ are given by $\mathbf{D}_{i,i} = -\lambda_i$. Upon computing $M_N(\kappa(u)^\gamma)$, we approximate $u$ by its projection onto $H_N$ if $u \notin H_N$. An important numerical issue to note is that we cannot use an approximation of $M_N(\kappa(u^N)^\gamma)$ obtained by spectral decomposition, that is, $M_N(\kappa(u^N)^\gamma) \approx \mathbf{U}^\top \kappa(\mathbf{D}_u)^\gamma \mathbf{U}$, for a diagonal matrix $\mathbf{D}_u$ with the diagonal entries are the eigenvalues of $M_N(u_N)$, and $\mathbf{U}$ is the orthonormal matrix. The reason is that the resulting matrix will not be in the form of \eqref{eq:Mn_d_more_than_1}. Instead, we could obtain $M_N(\kappa(u^N)^\gamma)$ matrix by applying Fourier transform directly to $\kappa(u^N)^\gamma$ and make use of \eqref{eq:Mn_d_more_than_1}. %
Writing \eqref{eqn:prior} as $\mathbf{v} = \mathbf{L}(u)^{-1}\mathbf{w}$, we obtain a composition of a Gaussian field from a unit Gaussian field given in \cite{Dunlop2018}. 

In what follows, for simplicity, with a slight abuse of notation for $\mathbf{L}(u)$, if $r \in H_N$ we also use $\mathbf{L}(\mathbf{r}) := \mathbf{L}(\sum_{l=-N}^{N}\widehat{r}(l)\phi_l) = \mathbf{L}(r)$.
Using this notation, the $J$ Gaussian field hyperpriors with zero mean assumption can be written in the following form:
\begin{align}
	\mathbf{L}(\mathbf{u}_{j-1}) \mathbf{u}_{j} &= \mathbf{w}_{j}. \label{eq:L_u_i_matrix}
\end{align}

The hyperprior layers consist of the random fields $u_0,\ldots, u_{J-1}$, where the random fields $u_0$ will be stationary. Within this multilayer hyperprior setting, the unknown field $\upsilon$ is equivalent to $u_J$.
With the assumption that each random field involved is real-valued, the number of element in $\mathbf{u}_{i}$ is only $N+1$ since the remaining element can be obtained by complex conjugation.
\section{Convergence analysis}\label{sec:Convergence}

We will show that  the solution of the original SPDE system  can be approximated with  Galerkin methods on $d$-dimensional torus $\mathbb{T}^d$. In our convergence analysis, we will restrict to $d\leq 3$, since we will rely on continuity of the sample paths (see Lemma \ref{lem:Hoelder}).  We start by carefully presenting the notation of  Galerkin method for analysis purposes. We  denote the constant $\kappa_0$ with 
 $\exp(u_{-1})$  for notational ease. The function $\kappa$ is taken to be a  smooth function, which is bounded from below and above  by  exponential functions. That is, 
 $c_1\exp(- a_1| x|)\leq \kappa(x),\kappa'(x)\leq c_2\exp(a_2 x)$ for some $c_1,c_2,a_1,a_2>0$ for all $x\in\mathbb R$.
   
Almost surely bounded functions  $u_{0},..., u_{J}$  are  a  weak solution of SPDE system
\begin{equation}\label{eqn:spde1}
 -\Delta u_{i} +  \kappa^{2}(u_{i-1})
u_{i} =  \beta_{i}^{1/2}\kappa^\nu (u_{i-1}) w_{i}, \; \text{ where }i=0,\dots, J, 
\end{equation}
if they satisfy 
\begin{equation} \nonumber
 	- \langle u_{i},  \Delta \phi\rangle   
	+  \langle \kappa^{2}(u_{i-1}) u_{i}, \phi\rangle  =  \beta^{1/2}_i  \sum_{p=-\infty }^\infty \widehat{w}_i(p) \langle  \kappa^{\nu}(u_{i-1})
	\phi_p ,  \phi\rangle,  \text{ where }  i=0,\ldots,J,
\end{equation}
 for all $\phi\in C^2(\mathbb T^d)$. Assuming that the functions $u_i$ are almost surely bounded guarantees that the inner product of $u_i$ and $\Delta \phi$ is well-defined as compared to  the inner product of $\nabla u_i$ and $\nabla \phi$, which  may not be well-defined. Here we  expressed  independent white noises $w_i$  on $\mathbb T^d$  with the help of an orthonormal  basis $\{\phi_p\}$ in $L^2(\mathbb T^d)$ as
random series  $w_{i}=\sum_{p=-\infty }^\infty \widehat{w}_i(p) \phi_p$, where the random coefficients  $\widehat{w}_i(p) \sim N(0,1)$ are independent.

In the first approximation step for the SPDE system \eref{eqn:spde1}, we will  approximate white noise  $w_{i}$  with its projections  $w^N_{i}(\mathbf{x})$ onto the subspace $H_N$.   That is, $
    w_i^N(\mathbf{x})=  \sum_{p=-N}^N  w_{i} \phi_p(\mathbf{x}).
  $ %
Then the Galerkin approximations $u_{i}^N$ of $u_{i}$ satisfy the system 
\begin{equation} 	\label{eq:WeakSolution_U0_N}
\langle \nabla u_{i}^N, \nabla  \phi\rangle 
	+  \langle \kappa^{2}(u_{i-1}^N)) u^N_{i}, \phi\rangle  =   \beta^{1/2}_i  \langle  \kappa^{\nu}(u^N_{i-1}) w^N_{i},\phi\rangle,  \; i = 0, \dots, J
\end{equation}
for all $\phi\in H_N$, where $u_{-1}^N(\mathbf{x}):= \ln (\kappa_0)$. Here we are allowed to use to use inner products of $\nabla u_i^{N}$ and $\nabla \phi$, since the approximated white noise $w^N_i$ belongs to 
$H_N$.
  
 We will denote with $L(u_{i-1})^{-1}$ the solution operator, which maps 
 $f$ from the negatively indexed Sobolev space $H^{-1}(\mathbb T^d)$  to the weak solution of
 $-\Delta u  +\kappa^2(u_{i-1}) u  = \beta^{1/2}_i\kappa^\nu (u_{i-1}) f$. 
Similarly, we will denote with $L_N(u_{i-1}^N)^{-1}$ the solution operator, which maps 
 $f\in H^{-1}(\mathbb T^d) $ to the Galerkin approximation $u^N\in H_N$ of the equation
$-\Delta u+\kappa^2(u_{i-1}^N)u = \beta^{1/2}_i \kappa^\nu(u_{i-1}^N)  f$. The matrix form of $L_N(u_{i-1}^N)^{-1}$ is given by $\mathbf L(\mathbf u_{i-1})^{-1}$ from Equation \eqref{eq:L_u_i_matrix}.
The solution operators  $L(u_{i-1})^{-1}$ and 
$L_N(u_{i-1}^N )^{-1}$ satisfy the following elementary norm estimates.
For simplicity, we will take $\beta_i=1$ from now on. 
\begin{lem}\label{lem:norm}
Let $u_{i-1}$ and $u_{i-1}^N$ be bounded functions and let  $\kappa$ be a  positive continuous function.
The mappings $L(u_{i-1}): L^2(\mathbb T^d) \rightarrow 
H^2(\mathbb T^d) $ and  $L_N(u_{i-1}^N): L^2(\mathbb T^d) \rightarrow 
H^2(\mathbb T^d) $ satisfy norm estimates
\begin{align*}
&\Vert L(u_{i-1})^{-1} \Vert_{L^2, H^2}
\leq  C \Vert \kappa(u_{i-1})\Vert_\infty ^\nu  \frac{\max(1,\Vert \kappa^2(u_{i-1})\Vert_\infty)}{ 
\min(1,\inf_{\mathbf x}\kappa^2(u_{i-1}(\mathbf{x})))^2}  \text{ and }\\
&\Vert L_N(u_{i-1}^N)^{-1}  \Vert_{L^2, H^2}
\leq  C \Vert \kappa(u_{i-1}^N )\Vert_\infty ^\nu  \frac{\max(1,\Vert \kappa^2(u_{i-1}^N )\Vert_\infty)}{ 
\min(1,\inf_{\mathbf x}\kappa^2(u_{i-1}^N (\mathbf{x})))^2} ,  \end{align*}
respectively.
\end{lem}
\begin{proof}
By the Lax-Milgram theorem \cite{Brenner2008}, 
$\Vert (-\Delta + \kappa^2(u_{i-1}) I )^{-1} \Vert _
{H^{-1},H^1} \leq C  /\min(1,$
$\inf \kappa^2 (u_{i-1})) $ and the multiplication operator has  norm  
$\Vert \kappa^\nu (u_{i-1})\Vert _{L^2,L^2}\leq  
\Vert \kappa(u_{i-1}\Vert_\infty^\nu$. Hence, the solution operator has norm
\begin{align*}
\Vert L(u_{i-1})^{-1} \Vert _{L^2,H^1}\leq C
\frac{\Vert \kappa(u_{i-1}(\mathbf{x}))\Vert_\infty^{\nu}}{\min(1,\inf_{\mathbf x}\kappa^2(u_{i-1}(\mathbf x)))}.
\end{align*} 
We rewrite the PDE  in the form
\begin{align*}
(-\Delta+ a ) u^f _{i} =  (a-\kappa^2(u_{i-1} )
u^f_{i} + \kappa^\nu(u_{i-1} ) f,  
\end{align*}
where the right-hand side belongs now to $L^2(\mathbb T^d)$. By inverting the operator $-\Delta+ a I$, we obtain  an equation for the  solution $u^f_{i}$, which leads to the norm estimate 
\begin{align*}\nonumber
\Vert u^f_{i}\Vert_{H^2} \leq & \frac{2}{\min(1,\inf_{\mathbf{x}}  \kappa^2(u_{i-1}(\mathbf{x})))} 
\left(
  \frac{C \Vert \kappa^2(u_{i-1})\Vert_\infty   \Vert \kappa(u_{i-1})\Vert _\infty^\nu   }{\min(1,\inf_{\mathbf{x}}\kappa ^2 (u_{i-1}(\mathbf{x}))) } +
 \Vert \kappa(u_{i-1})\Vert _\infty  ^\nu \right)  \Vert f\Vert_{L^2} \\
\leq & C \Vert \kappa(u_{i-1})\Vert_\infty ^\nu  \frac{\max(1,\Vert \kappa^2(u_{i-1})\Vert_\infty)}{ 
\min(1,\inf_{\mathbf{x}}\kappa^2(u_{i-1}(\mathbf{x})))^2} \Vert f\Vert_{L^2}
\end{align*}
after choosing  $a= \inf_{\mathbf x} \kappa^2 (u_{i-1}(\mathbf {x}) /2$.
Similar procedure leads to the desired estimate for $L_N(u_{i-1})$.
\end{proof}
We will later need the following technical lemma to establish convergence. 

\begin{lem}\label{lem:tech}
Let $u_{i-1}$ and $u_{i-1}^N$ be bounded functions and let  $\kappa$ be a continuously differentiable  positive function.
The mappings $L_N(u_{i-1}^N)^{-1}$ and $L(u_{i-1})^{-1}$ 
satisfy
\begin{align*}
\Vert  L_N  (u_{i-1}^N )^{-1} - L (u_{i-1})^{-1} 
\Vert_{L^2,L^2}  \leq   \frac{1}{N}  G_1(u_{i-1}) +   G_2(u_{i-1},u_{i-1}^N) \Vert u_{i-1}^N- u_{i-1}\Vert _{L^1}^{1/6},
\end{align*}
where 
\begin{align*}
G_1(u_{i-1}) &= C \frac{\max(1,\Vert \kappa^2(u_{i-1})\Vert_{\infty})^2 \max(  \Vert \kappa^{\nu }(u_{i-1} )\Vert_\infty, \Vert \kappa^{-\nu }(u_{i-1} )\Vert_\infty)}{\min(1, \inf \kappa^2(u_{i-1}\mathbf{(x)}))^3} \\    
 G_2(u_{i-1},u^N_{i_1}) &=C \frac{  \max(1,\Vert \kappa^\nu (u_{i-1})\Vert_\infty)
 \max  \left( \Vert \kappa  (u_{i-1}^N)\Vert _\infty,
\Vert\kappa   (u_{i-1})\Vert_\infty\right)^{5/3+5\nu/6}   }{\min(1,\inf_{\mathbf{x}} \kappa^2(u_{i-1}^N (\mathbf{x}))) \min(1,\inf_{\mathbf{x}} \kappa^2(u_{i-1} (\mathbf{x}))) }\\
& \times  \max_{t\in B}  \left(  \kappa'(t) \right)^{1/6} 
\end{align*}    
and the set $B$ is the interval $[ \min ( \inf_{\mathbf{x}} u^N(\mathbf x),\inf_{\mathbf{x}} u(\mathbf x)), \max(\Vert u^N\Vert_\infty, \Vert u \Vert _\infty )]$.
\end{lem}
\begin{proof}
We partition $T :=L_N (u_{i-1}^N )^{-1} - L (u_{i-1})^{-1}$ into two parts 
\begin{equation}\nonumber
T=
(L_N (u_{i-1}^N ) ^{-1}- L_N (u_{i-1})^{-1})
+(L_N (u_{i-1})^{-1} - L (u_{i-1})^{-1})= 
T_1+ T_2.
\end{equation}
By Cea's lemma \cite{Brenner2008}, the term $T_2$ has an upper bound
\begin{align*}
\Vert T_2 \Vert_{L^2,L^2} \leq C \frac{\max(1,\Vert \kappa^2(u_{i-1})\Vert_{\infty})}{\min(1, \inf_{\mathbf{x}} \kappa^2(u_{i-1}(\mathbf{x})))} \Vert (I- \widetilde P_N) L(u_{i-1})^{-1} \kappa^{-\nu }(u_{i-1}) \Vert_{L
^2,H^1} \Vert \kappa^{\nu }(u_{i-1}) \Vert_\infty,
\end{align*}
where $\widetilde P_N$ is the orthogonal projection onto $H_N$ in 
$ H^1$ and  $I$ is the identity operator. Let $f\in L^2$ and denote  $g:=L(u_{i-1})^{-1} \kappa^{-\nu }(u_{i-1} ) \widetilde  P_Nf$. Then  
\begin{align*}\nonumber
\Vert (I - \widetilde P_N)g \Vert_{H^1}^2
=& \sum_{|k_1|,|k_2|>n} (1+k_1^2+k_2^2)
|\widehat g_{k_1,k_2}|^2 
 =  \sum_{|k_1|,|k_2|>n} (1+k_1^2+k_2^2)^{-1} 
 |(-\Delta +1 )g)^{\wedge}|^2 _{k_1,k_2} \\ %
\leq& \frac{1}{n^2} \Vert (-\Delta + 1)g\Vert_{L^2}^2
\leq \frac{C}{N^2}  \Vert L(u_{i-1})^{-1}\Vert_{L^2,H^2}^2 \Vert \kappa(u_{i-1} (\mathbf{x}))^{-\nu}\Vert _\infty ^2 \Vert  f\Vert_{L^2}^2,
\end{align*}
where $-\Delta+1:H^2\rightarrow L^2$ is continuous, $g^\wedge $ denotes the Fourier transform and  $N=(2n+1)d$. Hence,
\begin{equation}\nonumber
\Vert  T_2 \Vert_{L^2,L^2}\leq \frac{C}{N} 
\frac{\max(1,\Vert \kappa^2(u_{i-1})\Vert_{\infty})^2}{\min(1, \inf \kappa^2(u_{i-1}\mathbf{(x)}))^3} \max(  \Vert \kappa^{\nu }(u_{i-1} )\Vert_\infty, \Vert \kappa^{-\nu }(u_{i-1} )\Vert_\infty).
\end{equation}
In the term $T_1$, we need to tackle the difference of 
$\kappa$-terms. We aim to use $L^p$-estimates in order to later  allow induction with respect to different layers. We partition  $T_1$ into simpler terms 
\begin{align*}
\Vert T_1\Vert_{L^2,L^2}  %
&\leq \Vert L_N(u_{i-1}^N)^{-1} \kappa^{-\nu} (u_{i-1}^N)  \left( \kappa^{\nu} (u_{i-1}^N)-\kappa^\nu (u_{i-1})\right)\Vert_{L^2,L^2}  \\
& + \Vert 
\left(L_N(u_{i-1}^N)^{-1} \kappa^{-\nu}(u_{i-1}^N) - L_N(u_{i-1})^{-1} \kappa^{-\nu}(u_{i-1})\right)  \kappa^{\nu} (u_{i-1})\Vert_{L^2,L^2} \\
&= \Vert T_{11}\Vert
+\Vert T_{12}\Vert.
\end{align*}
In the term  $T_{12}$, we apply the resolvent identity 
\begin{align*}
&L_N(u_{i-1}^N)^{-1} \kappa^{-\nu}(u_{i-1}^N)  = L_N (u_{i-1})^{-1} \kappa^{-\nu}(u_{i-1})\\ & +
L_N(u_{i-1})^{-1} \kappa^{-\nu}(u_{i-1})
(\kappa^{2}(u_{i-1})- 
\kappa^{2}(u_{i-1}^N )) L_N(u_{i-1}^N)^{-1} \kappa^{-\nu}(u_{i-1}^N), 
\end{align*}
which leads to
\begin{align*} %
\Vert T_{12}\Vert_{L^2,L^2} &\leq  \Vert  L_N(u_{i-1})^{-1} \kappa^{-\nu}(u_{i-1}) 
\Vert_{L^{3/2},L^2 }\Vert  \kappa^{2}(u_{i-1}) -  \kappa^{2}(u_{i-1}^N )  \Vert_{L^2,L^{3/2}}   \nonumber\\
& \times \Vert  L_N(u_{i-1}^N)^{-1} \kappa^{-\nu}(u_{i-1}^N) 
\Vert_{H^{-1},H^1} 
 \Vert  \kappa^{\nu} (u_{i-1})\Vert_\infty. 
\end{align*}
The space $L^{3/2}$ embeds continuously into $H^{-1}$, and 
$\Vert A\Vert_{L^{3/2},L^2}\leq  \Vert A\Vert_{H^{-1},H^1}$ for any operator $A$.  Hence by the Lax-Milgram theorem,
\begin{align*}
\Vert L_N(u_{i-1})^{-1} \kappa^{-\nu} (u_{i-1})\Vert_{L^{3/2},L^2}\leq 
\frac{1}{\min(1,\inf_{\mathbf{x}} \kappa^2(u_{i-1} (\mathbf{x}))) }.
\end{align*}
We are now treating $\kappa^{2}(u_{i-1}) -  \kappa^{2}(u_{i-1}^N )$ as a multiplication operator from  $L^2$ to $ L^{3/2} $. That is, a function $g$ defines a multiplication operator, which takes a function  $f$ to  the product of functions $g$ and $f$. 
By  H\"older's  inequality, any multiplication operator $g: L^2\rightarrow 
L^{3/2} $ has norm $\Vert  g\Vert_{L^6}$. Similar treatment for the term 
$T_{11}$ gives
\begin{align*}%
\Vert T_{11}\Vert_{L^2,L^2} %
\leq 
\Vert L_N(u_{i-1}^N)^{-1} \kappa^{-\nu} (u_{i-1}^N)\Vert_{L^{3/2},L^2} 
\Vert \kappa^{\nu} (u_{i-1}^N)-\kappa^\nu (u_{i-1})\Vert _{L^2,L^{3/2}}.
\end{align*}
The difference of  $\kappa$ terms in the estimates for $T_{11}$ and $T_{12}$ reduces to the difference of the functions $u_{i-1}$ through series of elementary estimates 
\begin{align*}
\Vert \kappa^a(u_{i-1}^N)  -\kappa^a (u_{i-1}
)\Vert_{L^6}^6&= 
\int_{\mathbb T^d} |\kappa^a( u_{i-1}^N(\mathbf{x}))- \kappa^a(u_{i-1}(\mathbf{x}))|^6 d\mathbf{x}  \\ &\leq      (2 \max\left( \Vert \kappa^a  (u_{i-1}^N)\Vert_\infty,  
\Vert \kappa^a (u_{i-1})\Vert_\infty\right))^5 \int  \left|\int_{u_{i-1}(\mathbf x) }^{u_{i-1}^N(\mathbf x)} \kappa'(t) dt \right|d\mathbf x \\
&\leq C \max  \left( \Vert \kappa^a  (u_{i-1}^N)\Vert^5 _\infty,
\Vert\kappa^a  (u_{i-1})\Vert_\infty^5\right)  \max_{t\in B}  \left(  \kappa'(t) \right)  \Vert u_{i-1}^N- u_{i-1}\Vert _{L^1},
\end{align*} 
where $a=\nu,2$ and  the set $B$ is the interval $[ \min ( \inf_{\mathbf{x}} u^N(\mathbf x),\inf_{\mathbf{x}} u(\mathbf x)), \max(\Vert u^N\Vert_\infty, \Vert u \Vert _\infty )]$.
\end{proof}
\begin{remark}\label{rem:expbounds}
When $\kappa$ is a continuously differentiable function with bounds $c_1\exp(- a_1 |t|) \leq \kappa(t),|\kappa'(t)| \leq 
c_2 \exp(a_2 |t|)$, where $c_1,c_2, a_1,a_2>0$, the functions $G_1$ and $G_2$ can be taken to be 
\begin{align*}
G_1(u_{i-1})&= C_1\exp(C_2(\Vert u_{i-1}\Vert_\infty))\\
G_2(u_{i-1},u_{i-1}^N)&= C_3 \exp(C_4 (\Vert u_{i-1}\Vert_\infty +  \Vert u_{i-1}^N\Vert_\infty).
\end{align*}
\end{remark}

The next lemma shows  that the SPDE system \eref{eqn:spde1} forces the  layers  to be H\"older-continuous.
\begin{lem}\label{lem:Hoelder}
Let $\kappa$ be a positive continuous function. 
Let $0<\alpha<1$ for $d=1,2$ and  $0<\alpha<1/2$ for $d=3$. The bounded weak solution $u_{i}$, $i=0,\ldots, J$, of the SPDE system \eref{eqn:spde1} is $\alpha$-H\"older continuous  with probability 1, and has the form 
\begin{equation}\label{eqn:Layer_form}
u_{i}(\mathbf{x})= \sum_{p=-\infty} ^\infty \widehat{w}_i(p)  L(u_{i-1})^{-1}\phi_p   (\mathbf {x}) =: 
L(u_{i-1})^{-1} w_{i} (\mathbf {x}), \; i=0,\dots,J.
\end{equation}
\end{lem}
\begin{proof}
We will show that if $u_{i-1}$ is continuous, then 
$u_{i}$ is H\"older-continuous. This will prove continuity inductively,  since the zeroth layer $u_{0}$ has constant $u_{-1}$.  
It is enough to verify continuity after conditioning with $u_{i-1}$,  since $\mathbb{P}(u_{i} \in C^{0,\alpha}(\mathbb T^d))=  \mathbb E [ \mathbb{P}(u_{i}\in C^{0,\alpha}(\mathbb T^d)  \mid u_{i})] =1$ with probability 1 if and only if $\mathbb P(u_{i}\in C^{0,\alpha}(\mathbb T^d) \mid u_{i-1})=1$ . After conditioning, $u_{i}$ will be a Gaussian field for which
we will apply Kolmogorov continuity criterium. To this end, we calculate
\begin{align*}
\mathbb E [|u_{i}(\mathbf{x})-u_{i}(\mathbf{x'})|^{2b} \mid  u_{i-1}] 
=  C_b \sup_{\Vert f \Vert_{L^2} \leq 1} |L(u_{i-1})^{-1} f(\mathbf{x})- L(u_{i-1})^{-1}f(\mathbf{x'})|^{2b},
\end{align*}
which follows from the  It\=o isometry and the equivalent way to calculate   $\Vert g\Vert _{L^2}$ as  $\sup_{\Vert f\Vert_{L^2}\leq 1}
\langle f,g \rangle$. %
When    $u_{i-1}$ is  continuous,  the function  $L(u_{i-1})^{-1} f (\mathbf{x})$ belongs to  $H^2$ by Lemma \ref{lem:norm}.  By Sobolev's  embedding theorem \cite{Evans2014},  $H^2$ embeds into 
$C^{0,\alpha}$ for $d\leq 3$. Hence,
\begin{align*}
  \mathbb E [|u_{i}(\mathbf{x})-u_{i}(\mathbf{x'})|^{2b} \mid  u_{i-1}] 
\leq  C  |\mathbf{x}-\mathbf{x'}|^{2b\alpha},
\end{align*}
which implies H\"older continuity with index smaller than  $(2b\alpha-d)/2b= \alpha-d/2b$, where $b$ can be arbitrarily  large.  

To complete the proof for existence of the solution, we insert  the solution candidate \eref{eqn:Layer_form} into the SPDE system  \eref{eqn:spde1} and do direct calculations. Uniqueness of the solution follows from Lax-Milgram theorem.

\end{proof}

We will show that  $u^N_{i}$ converges in probability to 
$u_{i}$. The proof uses  uniform tightness of the distributions of $u_{i}^N-u_{i}$, which is  a necessary condition for the convergence.   We recall  sufficient conditions for  the uniform tightness (see  p. 61 in \cite{Bogachev_Weak}).
\begin{lem}\label{lem:tight}
The random fields  $U^N$ are uniformly tight
on $C(\mathbb T^d)$ if and only if there exists a function $K:C(\mathbb T^d)\rightarrow [0,\infty)$ with the following 
properties.
\begin{itemize}
\item[(1)] The set $\{ g\in C(\mathbb T^d): K(g)\leq C\}$ is
compact for any $C>0$,
\item [(2)]  $K(U^N) <\infty$ almost surely  for every $N$, and  
\item [(3)]  $\sup_N \mathbb E[ K(U^N)] <\infty$.
\end{itemize}
\end{lem}

We will  need several iterations of the logarithm so we define the iterated 
composition by setting $F(x)=\ln(1+x)$, 
$F_0(x)=x$ and $F_{n+1}(x)=F\circ F_n(x)$. 
\begin{remark}\label{rem:subadd}
The function $F_i$ is increasing and, moreover,  subadditive on non-negative numbers. That is, $F_n(x+y)\leq F_n(x)+F_n(y)$ for all $x,y\geq$, which follows by induction from subadditivity  
$\ln(1+x+y)\leq \ln((1+x)(1+y))= \ln(1+x)+ \ln(1+y)$ of each $F$. Similar procedure shows that $F_n(xy)\leq F_n(x)+ F_n(y)$.
\end{remark}

\begin{lem}\label{lem:ut}
Let $d=1,2$ or $3$. Let $\kappa$ be a continuous function with bounds $c_1 \exp(- a_1 |t|) \leq \kappa (t) \leq 
c_2\exp(a_2 |t|)$, where $c_1,c_2,a_1,a_2>0$. Let $u_{i}^N$, $i=0,\ldots,J$, solve 
the system \eref{eq:WeakSolution_U0_N} and let $u_i$, $i=0,\ldots,J$ solve the system \eqref{eqn:spde1}.  Then  the random fields  $u_{i}^N$  are uniformly tight on $C(\mathbb T^d)$, the random fields  $u_{i}^N-u_{i}$  are uniformly tight on $C(\mathbb T^d)$, 
and the vector-valued random fields $(u_{i}^N,u_{i})$ are uniformly tight on $C(\mathbb T^d;\mathbb R^2)$.
\end{lem}
\begin{proof}
We equip the H\"older space $C^{0,\alpha}(\mathbb T^d;\mathbb R^p)$ with its usual norm \begin{equation}\nonumber
\Vert g\Vert _\alpha = \sup_{x\not= y} 
\frac{|g(x)-g(y)|}{|x-y|^\alpha} +\sup_{x} |g(x)|.
\end{equation}
Here $\alpha$ is chosen as in Lemma \ref{lem:Hoelder}. We will use Kolmogorov-Chentsov tightness criterium (see \cite{Kallenberg})
to show the  uniform tightness
of the  zeroth order  layers, which are Gaussian. The desired estimate 
\begin{align*}
\mathbb E[ |u_{0}^N(\mathbf{x})-u_{0}^N(\mathbf{x'})|^a] \leq C |\mathbf{x}-\mathbf{x'} |^{d+b}
\end{align*} 
follows from choosing large enough $a$ in 
\begin{align*}
\mathbb E[ |u_{0}^N(\mathbf{x})-u_{0}^N(\mathbf{x'})|^{a}] &\leq C  \sup_{\Vert f\Vert_{L^2}\leq 1}  
\frac{|L_N(\ln(\kappa_0))^{-1}
f (\mathbf{x}) - L_N(\ln(\kappa_0))^{-1}f (\mathbf{x'}) |^{a}}{|\mathbf{x}-\mathbf{x'}|^{a\alpha}}  |\mathbf{x}-\mathbf{x'}|^{a\alpha} \\
& \leq  C\Vert L_N(\ln(\kappa_0))^{-1}\Vert_{L^2,C^{0,\alpha}}^a |\mathbf{x}-\mathbf{x'}|^{a\alpha} ,
\end{align*} 
where we applied It\=o isometry and the definition of $L^2$-norm as a supremum.  By Lemma \ref{lem:norm} and  Sobolev's embedding theorem, the operator norm of 
$L_N(\ln(\kappa_0))^{-1}$  is bounded for any $0<\alpha<1/2$.  Since     $\kappa_0$ is a constant, the bound is uniform. The case for 
$u_{0}^N -u_{0}$ and $(u_{0}^N,u_{0})$  follow similarly with the help of the triangle inequality. 

For other layers, we use Lemma \ref{lem:tight}, where we 
choose  $K(g)= F_i(g) $. For simplicity, we demonstrate Condition 1 only for  $i=3$, since the generalization is clear. The set
\begin{equation}\nonumber
\{g: K(g)\leq C\} =\{ g: \Vert g\Vert _{\alpha}  \leq \exp(\exp(\exp(C) -1)-1)-1 \}
\end{equation}
is clearly a closed set, which  contains bounded equicontinuous  functions.   The set $A\subset C(\mathbb T^d)$ is then compact by  the Arzel\'a-Ascoli theorem (see \cite{Rudin1976}).  Moreover, the fields   $u_{i}^N-u_{i}$ are almost surely $\alpha$-H\"older continuous, by Lemma \ref{lem:Hoelder} and  since the approximations belong to 
$H_N$. Hence,  Condition 2  holds. 
To show Condition 3, we write 
$u_{i}^N$  as $L_N (u_{i-1}^N)^{-1}  w_{i}^N$ and 
$u_{i}$ as  $L(u_{i-1})^{-1} w_{i}$. By the triangle inequality and the subadditivity of $F_i$ (see Remark \ref{rem:subadd}), we can 
check the boundedness for  $u_{i}^N$ and $u_{i}$ separately. Since the  procedure is the same for both of the terms, we only show here the case for $u_{i}^N$. 
By Jensen's inequality
\begin{equation}\label{eqn:jensen}
\mathbb E[\ln(1+ \Vert u_{i}^N\Vert_{\alpha}) \mid 
 u_{i-1}^N] \leq  
  \ln \left(
 \mathbb E[ 1+ \Vert L_N(u_{i-1}^N)^{-1} w_{i}^N \Vert_{\alpha}  
\mid  u_{i-1}^N]  \right).
\end{equation}
The conditioning with $u_{i-1}^N$ lets us compute expectations of Gaussian variables. Instead of attacking directly the expectation in Equation \eref{eqn:jensen}, we will seek a Gaussian zero mean random field $U$ with larger variance than the conditioned $u_{i}^N$.  Then also certain other expectations of $u_i^N$ will be bounded by expectations of $U$. Under conditioning, the variances of the  random variables
 $\int\phi(\mathbf{x})  L_N(u_{i-1}^N)^{-1}w^N_{i}  (\mathbf {x}) d\mathbf{x}$ are 
 \begin{equation}\nonumber %
\mathbb E[ \langle w_{i}^N , (L_N (u_{i-1}^N)^{-1})^* \phi\rangle ^2 \mid  u_{i-1}^N] = 
\Vert P_N (L_N (u_{i-1}^N)^{-1})^* \phi \Vert _{L^2}^2  \leq 
\Vert L_N (u_{i-1}^N)^{-1}\Vert_{L^2,H^2} ^2 \Vert \phi\Vert_{H^{-2}}^2,  
\end{equation}
where the last inequality follows from properties of the adjoint operator
and  Lemma \ref{lem:norm}.  Set $U$ to be a zero mean Gaussian random field $U$ on $\mathbb T^d$ whose
covariance is defined by equations  
$\mathbb E[\langle U,\phi\rangle ^2]= \Vert \phi\Vert_{H^{-2}}^2$ for all smooth $\phi$. 
Then $U$ has  sample paths in 
H\"older space  $C^{0,\alpha}(\mathbb T^d)$ by Sobolev's embedding theorem and   Kolmogorov's continuity theorem (see \cite{Marcus2006}). By  Fernique's theorem (e.g. \cite{Bogachev_Gaussian}), the expectation $\mathbb E[ \Vert U\Vert_{\alpha}]$  is finite.  
Since  norms  are  absolutely continuous functions,   also  the conditional  expectation of $\Vert u_{i}^N\Vert_{\alpha}$ is bounded by  $\Vert L(u_{i-1}^N) ^{-1}\Vert_{L^2,H^2}\mathbb E[\Vert U\Vert _{\alpha}]$  (see   Corollary 3.3.7  in \cite{Bogachev_Gaussian}).
 
Further application of  Lemma \ref{lem:norm}  in Equation \eref{eqn:jensen} gives our main estimate
\begin{align}
\mathbb E[ \ln ( 1+ \Vert  u_{i}^N\Vert _\alpha)  \mid
u_{i-1}^N] 
&\leq   \ln \left [1+ 
 c \Vert \kappa(u_{i-1}^N (\mathbf{x}))\Vert_\infty ^\nu  \frac{\max(1,\Vert \kappa^2(u_{i-1}^N (\mathbf{x}))\Vert_\infty)}{ 
\min(1,\inf\kappa^2(u_{i-1}^N (\mathbf{x})))^2}\mathbb E[\Vert U \Vert _{\alpha}] \right] \nonumber \\
& \leq  c_{\alpha,\nu} (1+ \Vert u_{i-1}^N\Vert_\infty ), \label{eqn:main_est}
\end{align} 
where  we applied the bounds of $\kappa$ and the elementary inequalities  
$\max(a,b\exp(c))
\leq \max(a,b)\exp(|c|)$,  $\min(1,\inf \exp(g(x)))\geq 
\exp(-\Vert g\Vert_\infty)$, $ (1+ab)\leq (1+a)(1+b)$ and  $1+a\leq 2\max(1,a)$. Here we can choose  constants larger than 1. 

When $i=1$,  the expectations  \eref{eqn:main_est}  are uniformly bounded,  since   expectations of $\Vert u_{0}^N\Vert_\infty$  are bounded.  Then 
Lemma \ref{lem:tight} shows that $u_{1}^N$
are uniformly tight. For the subsequent layers  we need to operate multiple times with the logarithm and  Jensen's inequality through inductive steps
\begin{align*}
\mathbb E[ F_i( \Vert  u_{i}^N\Vert_\alpha)) ]
&\leq \mathbb E[F_{i}(\mathbb E[\Vert  u_{i}^N\Vert_\alpha\mid  
u_{i-1}^N])]
\leq \mathbb E [ F_{i}(c(1+\Vert u_{i-1}^N\Vert_\alpha))]\\
& \leq C+ \mathbb E[F_{i-1}(\Vert u_{i-1}^N\Vert_\alpha))]
\end{align*}
with the help of the additivity properties of $F_i$ from  Remark \ref{rem:subadd}.

 \end{proof}

\begin{thm}\label{thm:main}
Let $d=1,2,$ or $3$. Let $\kappa$ be a continuously differentiable  function with bounds $c_1 \exp(-a_1| t|) \leq \kappa (t), |\kappa'(t)| \leq 
c_2\exp(a_2| t|)$, where $c_1,c_2,a_1,a_2>0$. Let $u_{i}^N$, $i=0,\ldots,J$.
The solution  $(u_{0}^N, \ldots, u_{J}^N)$  of the Galerkin system \eref{eq:WeakSolution_U0_N} converge in probability 
to the weak solution  $(u_{0},\ldots, u_{J})$ of the SPDE system \eref{eqn:spde1} on $L^2(\mathbb T^d;\mathbb R^{J+1})$  as $N\rightarrow \infty$. 
\end{thm}
\begin{proof}
It is enough to show componentwise convergence. 
Uniform tightness on $C(\mathbb T^d)$ implies uniform tightness on $L^2(\mathbb T^d)$, which in turn implies  relative compactness in weak topology of distributions. Hence, by Lemma \ref{lem:ut} each subsequence of the distributions of   $u_{i}^N - u_{i}$ has a weakly convergent subsequence, say $u_{i}^{N_{k}} - u_{i}$. Recall, that the convergence in probability is equivalent to the convergence of 
$
\tau_{N}:=\mathbb E[\min(1,\Vert u_{i}^N -u_{i}\Vert_{L^2})],
$
where $\min(1,\Vert\cdot\Vert_{L^2})$ is now  a bounded continuous function. By weak convergence of distributions, the subsequence  $\tau_{N_k}$  has  some limit. It remains to verify that limits of  $u_{i}^{N_k}-u_{i}$  in distribution are  zero.  The characteristic functions of 
$u_{i}^N-u_{i}$ converge to 1, if 
$
\tau_N(\phi) := \mathbb E [ \min(1, |\langle  u_{i}^N -u_{i},\phi\rangle|),
$
converge to zero  for every $\phi\in L^2$.   This formulation   makes calculation of  expectations manageable.   The essential  difference to other approaches  arises from the monotonicity of conditional expectations. Namely,  the  property  $\mathbb E[ 1- \min(1, G) |\Sigma_0]\geq 0 $ for $G\geq 0$ implies that 
\begin{equation}\label{eqn:mono}
\mathbb E[ \min(1, G) \mid \Sigma_0]
= \min(1, \mathbb E[ \min(1, G) \mid \Sigma_0])
\leq \min(1, \mathbb E[G \mid \Sigma_0]).
\end{equation}
Especially,
\begin{equation}\label{eqn:F1}
\tau_N(\phi) \leq  \mathbb  E [ \min(1, \mathbb E [ |\langle  u_{i}^N -u_{i},\phi\rangle | \mid  w_{0}, \ldots,w_{i-1}  ])],
\end{equation}
where we applied \eref{eqn:mono} after taking a conditional expectation inside the expectation. We 
condition with white noises in order to handle both 
$u_{i}$ and its approximation $u_{i}^N$ easily.
Moreover, the cutoff function $\min(1,\cdot)$ is nondecreasing and subadditive.  
 
Under conditioning with $w_{0}, \ldots,w_{i-1}$, the 
random fields  $u_{i}^N$ and $u_{i}$ in Equation  \eref{eqn:F1}  become 
Gaussian, which enables us  to   compute
\begin{align*}
\tau_N(\phi)
=  & \mathbb E [ \min (1, c\Vert ( P_N (L_N(u_{i-1}^N)^{-1})^*  -(L(u_{i-1})^{-1})^* ) \phi\Vert_{L^2} ) ]  \\
\leq&   \mathbb E\left[ \min\left(1,   c_\phi  \Vert  L_N(u_{i-1}^N)^{-1} -L(u_{ i-1})^{-1}\Vert_{L^2,L^2} \right)+ \min\left(1, 
c \Vert  (P_N-I) (L (u_{i-1})^{-1})^*\phi \Vert_{L^2} \right)\right]\\ =&:\tau_N^1(\phi) +
\tau_N^2(\phi)
\end{align*}
via  the It\=o isometry and  the properties of adjoints. Since $L (u_{i-1})^*\phi\in L^2 $, the term $\tau_N^2(\phi)$ converges to zero. %
We will apply  Lemma \ref{lem:tech} for the difference of operators $L_N(u_{i-1}^N)^{-1}$ and $L^{-1}(u_{i-1})$ in  $\tau_N^1(\phi)$, which leads to  a well-behaving estimate 
\begin{align*}
\tau_N^1(\phi)\leq &\mathbb E \bigg[ \min\left( 1,   
 \frac{C_\phi }{N}  \exp(C_1\Vert u_{i-1}\Vert_\infty )\right) \bigg ]
+ \mathbb E[ \min(1,C_\phi \exp( C_2 (\Vert u_{i-1}\Vert_\infty  +  \Vert u_{i-1}^N \Vert_\infty ))\\
\times &  \Vert u_{i-1}^N - u_{i-1}\Vert _{L^2}^{1/6})] =:\tau_N^{11}(\phi)+\tau_N^{12}(\phi).
\end{align*}
The first term $\tau_N^{11}(\phi)$ converges to zero by Lebesgue's dominated convergence theorem.   For the term $\tau_N^{12}(\phi)$,   we will use  the uniform tightness  of the the vector-valued random fields $(u^{N}_{i-1},u_{i-1})$ shown in Lemma \ref{lem:ut}.  Let  $K_{i-1}=K_{i-1}(\epsilon) \subset C(\mathbb T^d;\mathbb R^2)$ be a compact set for which  $P((u_{i-1}^N, u_{i-1}) \in K^C)<\epsilon$. Then  
\begin{align*}
\tau_N^{12}(\phi) &= 
\mathbb E[\min( 1,C_\phi\exp(C_2 (\Vert u_{i-1}\Vert_\infty  +  \Vert u_{i-1}^N \Vert_\infty)) \Vert u_{i-1}^N - u_{i-1}\Vert _{L^2}^{1/6}) (1_{K}+1_{K^C})]\\
&\leq \mathbb E[\min(1, C_\epsilon \Vert u_{i-1}^N-u_{i-1}\Vert_{L^2})]^{1/6}   +  \mathbb E [1_{K^C}]
\end{align*}
by Lemma \ref{lem:tech},  Remark \ref{rem:expbounds} and 
Jensen's inequality.
Thus $\tau_N^{12}(\phi)$ converges to zero if $u_{i-1}^N-u_{i-1}$ 
converge to zero in probability on $L^2$. 

When $i=0$,  the above procedure shows then that $u_{0}^N$ converges to $u_{0}$ in probability,  because the sublayers are then constants.  
By induction, $u_{i}^N $ converges in probability to  $u_{i}$. 
\end{proof}

Prohorov's theorem \cite{Kallenberg} hands us weak convergence on 
the space of continuous functions. 
\begin{cor}
The distributions of  $u_{i}^N$ converge weakly to 
the distribution of $u_{i}$ on $C(\mathbb T^d)$ and the joint distribution of $(u_0^N,\ldots,u_J^N)$ converges weakly to the joint distribution of $(u_0,\ldots,u_J)$
 on $C(\mathbb T^d, \mathbb R^{J+1})$. 
\end{cor}
\begin{proof}
The random fields $u_{i}^N$ are tight on 
$C(\mathbb T^d)$ by Lemma \ref{lem:ut}. By Prohorov's theorem, the closure of their distributions forms a sequentially compact set, implying the existence of weak limit. By Theorem \ref{thm:main}, each weakly converging subsequence of the distributions has the same limit, that is, the distribution of $u_{i}$.  Indeed, since convergence in probability on $L^2(\mathbb T^d)$ implies weak convergence on $L^2(\mathbb T^d)$, the characteristic functions 
$
\mathbb E[\exp(i\langle u^N_{i}, \phi\rangle )]
$ 
converge to $\mathbb E[\exp(i\langle u_{i}, \phi\rangle )]$ for all $\phi\in L^2(\mathbb T^d)$. The set of bounded continuous functions $u\mapsto \exp(i\langle u, \phi\rangle)
$, where $\phi$ are smooth on $\mathbb T^d$, separate the functions in $C(\mathbb T^d)$. Hence, the limits of these characteristic functions are enough to identify the weak limit on $C(\mathbb T^d)$. The joint distribution is handled similarly.
\end{proof}

We recall a posterior convergence result from
 \cite{Lasanen2012} with  notation used in 
 \cite{Stuart2010}.
\begin{thm}
Let the posterior distribution $\mu^\mathbf{y}$ of an unknown  $u$ given an observation $\mathbf{y}$ have the Radon-Nikodym density 
\begin{align*}
\frac{d\mu^{\mathbf y}}{d\mu^0} (u)
\propto  \exp(-\Phi(u,\mathbf{y})) 
\end{align*}
with respect to the prior distribution $\mu^0$ of $u$, where 
 $\Phi(\cdot, \mathbf{y}) \geq -C_{\mathbf y}$.  Let $u^N$ be an approximation of $u$ and let the posterior distribution of $u^N$ given an observation   $\mathbf{y}_N$
have also the Radon-Nikodym density 
\begin{align*}
\frac{d\mu^{\mathbf y_N}_N }{d\mu^0_N} (u)
\propto  \exp(-\Phi(u, \mathbf{y}_N)) 
\end{align*}
with respect to the prior distribution $\mu^0_N $ of $u^N$. 

If the prior distributions  $\mu^0_N$ converge  weakly to $\mu^0$, then the posterior distributions  $\mu_N^\mathbf{y}$ converge weakly to $\mu^\mathbf{y}$. 
\end{thm}
Especially, the joint prior distribution of 
$(u_0^N,\ldots,u_J^N)$ converge weakly on $C(\mathbb T^d, 
\mathbb R^{J+1})$ to the joint distribution of $(u_0,\ldots, u_J)$. Hence, the corresponding posteriors converge also weakly.

\section{Bayesian inference algorithm}\label{sec:Sampling}
In this section, we will develop a Bayesian inference procedure to sample the Fourier coefficients from a posterior distribution where the prior is given by a multi-layered Gaussian fields. Let us work directly in the Fourier coefficient $\mathbf{u}_J$ of $u_J$, and denote by $\mu^0(d\mathbf{u}_J)=\mathbb{P}(d\mathbf{u}_J)$ and   $\mu^{\mathbf{y}}(d\mathbf{u}_J)  = \mathbb{P}(d \mathbf{u}_J\mid\mathbf{y})$ the prior and the posterior distribution  of $u_J$, the unknown target field, respectively, when the measurement is given by
\begin{align}
\mathbf{y}=\mathbf{H}\mathbf{u}_J+ \mathbf{e}.\label{eq:Linear_measurement}
\end{align}
In this equation, $\mathbf{y}$ is a vector which contains all of the measurements. The elements of the matrix $\mathbf{H}$ correspond to the linear mappings $\{h_k\}$ for the respective Fourier components. With the number of measurements taken is given by $m$,  the measurement noise $\mathbf{e}$ is an $m$-dimensional  Gaussian random vector with zero mean and covariance $\mathbf{E}$. The posterior distribution $\mu^{\mathbf{y}}$ will be absolutely continuous with respect to the prior $\mu^0$, and the density of the posterior with respect to the prior  is given by
\begin{align}
\frac{d\mu^{\mathbf{y}}}{d \mu^0} (\mathbf{u}_J) =& \frac{1}{Z}\exp\left(-\Phi(\mathbf{u}_J,\mathbf{y})\right),
\end{align}	
where $\Phi(\mathbf{u}_J,\mathbf{y}) :=  \frac{1}{2}\left\| \mathbf{E}^{-1/2}(\mathbf{y}-\mathbf{H}\mathbf{u}_J)\right\|^2$ is the potential function and the normalization constant $Z :=  \int \exp(-\Phi(\mathbf{u}_J,\mathbf{y})) \mu^0(d \mathbf{u}_J)$.
The posterior probability $\mu^{\mathbf{y}}(d\mathbf{u}_J)$ can be written in the following form
\begin{equation}
\mu^{\mathbf{y}}(d\mathbf{u}_J) \propto \exp\left(-\Phi(\mathbf{u}_J,\mathbf{y})\right) \mu^0(d\mathbf{u}_J). \label{eq:posterior}
\end{equation}

In the next section, we describe the MCMC algorithm to sample from the posterior distribution. 

\subsection{Non-centered algorithm}
Consider the case of hierarchical prior distribution $\mu(d\mathbf{u}_J)$  with $J$  hyperprior layers which we construct as
\begin{subequations}
	\begin{align}
		\mathbf{L}_0 \mathbf{u}_0 &= \mathbf{w}_0,\\
		\mathbf{L}(\mathbf{u}_{j-1}) \mathbf{u}_{j} &= \mathbf{w}_{j}, \quad j = 1, \cdots, J,
	\end{align}\label{eqs:Lu_relations}		
\end{subequations}
where $\mathbf{u}_j$ contains the Fourier coefficients of the field $u_j$, for $j=0,\ldots,J$. 
We fix $g(x) = \exp(-x)$. Since $u_{-1} = \ln(\kappa_0)$ is a constant, we can write, $\mathbf{L}_0 = \mathbf{L}( \mathbf{u}_{-1} )$, where $k$-th element of $\mathbf{u}_{-1}$ is equal to $\ln(\kappa_0) \delta_{k,0}$, where $\delta_
{k,0}$ is the Kronecker delta. By \eqref{eqs:Lu_relations}, we can define a linear transformation from $\mathbf{w}_j$ to $\mathbf{u}_j$ for $j>0$ as follows: 
	\begin{align}
		\mathbf{u}_j = \tilde{U}(\mathbf{w}_j,\mathbf{u}_{j-1}) &:= \mathbf{L}(\mathbf{u}_{j-1})^{-1}\mathbf{w}_{j}. \label{eq:U_w_u}
	\end{align}

Using \eref{eq:U_w_u} we can define a transformation from $\mathbf{w}$ to $\mathbf{u}$ as follows
\begin{align}
	\mathbf{u} = U(\mathbf{w}) &=  (\tilde{U}(\mathbf{w}_0,\mathbf{u}_{-1}),
	\tilde{U}(\mathbf{w}_1,\cdot) \circ \tilde{U}(\mathbf{w}_0,\mathbf{u}_{-1}),
	\ldots,
	\tilde{U}(\mathbf{w}_{J},\cdot)\circ \ldots \circ \tilde{U}(\mathbf{w}_0,\mathbf{u}_{-1})). \label{eq:U}
\end{align}
The dependence of $\mathbf{u}_{j}$ on $\mathbf{u}_{j-1}$, $j = 1,\ldots,J$, leads us to 
\begin{equation}
\mu^0 (d \mathbf{u}_J) = \mathbb{P} \left(d \mathbf{u}_0\right) \prod_{j=1}^{J} \mathbb{P}
\left(d\mathbf{u}_{j}\mid \mathbf{u}_{j-1}\right). \label{eq:prior}
\end{equation}

When evaluating the posterior function \eqref{eq:posterior}, it is necessary to compute the log determinant of $\mathbf{L}(\mathbf{u}_{j-1})$ in $\mathbb{P}
(d\mathbf{u}_{j}| \mathbf{u}_{j-1})$ for each $1<j\leq J$, respectively. These calculations are expensive in general \cite{dong2017}. There is also a singularity issue if we sample directly from $\mathbb{P}(d\mathbf{u}|\mathbf{y})$ if $N$ approaches infinity \cite{Dunlop2018}. We can avoid these issues by using the reparametrization \eref{eq:U_w_u}, where instead of sampling the Fourier coefficients $\mathbf{u}$,  we sample the Fourier coefficient of the noises $\mathbf{w}$, which is then called \emph{non-centered} algorithm \cite{Papaspiliopoulos2007,Yu2011}. That is, we can write:
	\begin{align}
	\dfrac{d\tilde{\mu}^{\mathbf{y}}}{d \tilde{\mu}^0} (\mathbf{w}) =& \frac{1}{\tilde{Z}}\exp\left(-\tilde{\Phi}(\mathbf{w},\mathbf{y})\right) := \frac{1}{\tilde{Z}}\exp\left(-\Phi(U(\mathbf{w}),\mathbf{y})\right).
	\end{align}\label{eq:non_centered}
In this equation, $\tilde{\mu}^0 := \mathbb{P}(d\mathbf{w})$ and $\tilde{\mu}^{\mathbf{y}}:= \mathbb{P}(d\mathbf{w}|\mathbf{y})$ are the prior and posterior of the $\mathbf{w}$, respectively. The preconditioned Crank-Nicolson (pCN) algorithm \cite{Cotter2013} can be used to sample from $\mathbb{P}(d\mathbf{w}|\mathbf{y})$, and it is well defined even for the case of $N$ goes to infinity, using the fact that the prior for $\mathbf{w}$ is a standard Gaussian distribution. One implementation of the pCN algorithm is given by Algorithm \ref{Noncentered_algorithm}. Due to linearity assumption of the forward model \eref{eq:Linear_measurement}, we can leverage the standard Gaussian regression in addition to the non-centered reparametrization. This procedure has been proposed in \cite{Dunlop2018} for general deep Gaussian fields. Here, we adapt this algorithm for our Galerkin method. The resulting algorithm is a Metropolis within Gibbs type \cite{DaniGamerman2006} where the Fourier coefficients $\{\mathbf{u}_j\}, j=0,\ldots,J-1$ are sampled using a PCN algorithm via reparametrization \eqref{eq:U_w_u}, and the Fourier coefficients for the last layer  $\mathbf{u}_J$ are sampled directly. The detail is given as follows. By marginalization of $\mathbf{u}_{J-1}$, we can write 
$
	\mathbb{P}(d\mathbf{u}_{J}|\mathbf{y}) = \int \mathbb{P}\left(d\mathbf{u}_{J}\mid\mathbf{u}_{J -1},\mathbf{y}\right) \mathbb{P}\left(d\mathbf{u}_{J -1}\mid\mathbf{y}\right).
$ Furthermore, from the standard Gaussian regression, we can sample directly from $\mathbb{P}(d\mathbf{u}_{J}|\mathbf{u}_{J -1},\mathbf{y})$ by using:
\begin{align}
	V(\mathbf{u}_{J -1},\mathbf{y}) :=& \begin{pmatrix}  \mathbf{E}^{-1/2} \mathbf H \\ \mathbf L(\mathbf{u}_{J -1})\end{pmatrix}^\dagger \left(\begin{pmatrix}\mathbf{E}^{-1/2} \mathbf y \\ 0 \end{pmatrix} + \tilde{\mathbf{v}}\right), \label{eq:V_sample}\\
	\mathbf{u}_{J -1} =& \tilde{U}(\mathbf{w}_{J-1},\cdot)\circ \cdots \circ \tilde{U}(\mathbf{w}_0,\mathbf{u}_{-1}),\nonumber\\
	\tilde{\mathbf{v}} \sim& N(0,\mathbf{I}).\nonumber
\end{align} 

 Writing $\mathbf{y} = \mathbf{H}\mathbf{L}(\mathbf{u}_{J -1})^{-1}\mathbf{w}_{J} + \mathbf{e}$, the conditional probability density of $\mathbf{y}$ given $\mathbf{u}_{J -1}$ is given by $p(\mathbf{y}|\mathbf{u}_{J -1}) = {N}(\mathbf{y}| 0,\mathbf{H}\mathbf{L}(\mathbf{u}_{J -1})^{-1}\mathbf{L}(\mathbf{u}_{J -1})^{-\top}\mathbf{H}^\top+\mathbf{E})$. Let $\bar{\mathbf{w}} = (\mathbf{w}_0,\ldots,\mathbf{w}_{J-1})$. To obtain samples from $\mathbb{P}(d\mathbf{u}_{J-1} |\mathbf{y})$ we can use reparametrization \eqref{eq:U_w_u} and  Algorithm \ref{Noncentered_algorithm} to sample from $\mathbb{P}(d\bar{\mathbf{w}} |\mathbf{y})$. The probability distribution $\mathbb{P}(d\bar{\mathbf{w}} |\mathbf{y})$ is given as follows:
\begin{subequations}
	\begin{align}
	\mathbb{P}(d\bar{\mathbf{w}} |\mathbf{y}) \propto& \exp\left(-\tilde{\Psi}(\bar{\mathbf{w}},\mathbf{y})\right) \mathbb{P}(d\bar{\mathbf{w}}),\\
	\tilde{\Psi}(\bar{\mathbf{w}},\mathbf{y}) := \Psi\left(U(\bar{\mathbf{w}}),\mathbf{y}\right) =& \frac{1}{2}\norm{\mathbf{y}}_{\mathbf{Q}}^2 + \frac{1}{2}\log \det(\mathbf{Q}),\\
	\mathbf{Q} =& \mathbf{H}\mathbf{L}(\mathbf{u}_{J -1})^{-1}\mathbf{L}(\mathbf{u}_{J -1})^{-\top}\mathbf{H}^\top+\mathbf{E}.
	\end{align}
\end{subequations}
To sample from $\mathbb{P}(d\bar{\mathbf{w}} |\mathbf{y})$ using Algorithm \ref{Noncentered_algorithm}, we use $J -1$, and $\tilde{\Psi}$ for $J $ and $\tilde{\Phi}$, respectively. 

\begin{algorithm}[h]
	\SetKwData{Layer}{Layer}
	\SetKwData{logRatio}{logRatio}
	\SetKwData{accepted}{accepted}
	\SetKwFunction{draw}{draw}
	\SetKwFunction{sample}{sample}
	\SetKwFunction{logDet}{logDet}
	\SetKwFunction{constructLMatrix}{constructLMatrix}
	\DontPrintSemicolon
   \SetKwInOut{Input}{Input}
   \SetKwInOut{Output}{Output}
   \Input{$J,\mathbf{w},\mathbf{y},\tilde{\Phi}(\mathbf{w},\mathbf{y})$}
	\Output{\accepted, $\mathbf{w}$}
	\accepted = $0$\;
	\draw $\mathbf{w}'\sim N(0,\mathbf{I})$\;
	$\tilde{\mathbf{w}} = \sqrt{1-s^2}\mathbf{w} + s \mathbf{w}'$\;
	\logRatio = $\tilde{\Phi}(\mathbf{w},\mathbf{y})-\tilde{\Phi}(\tilde{\mathbf{w}},\mathbf{y})$\;
	\draw $\omega \sim \text{Uniform}[0,1]$\;
	\If{\logRatio $>\ln(\omega)$}{
	    $\mathbf{w}  = \tilde{\mathbf{w}}$\;
		\accepted = 1\;
	}
	
	\caption{preconditioned Crank-Nicolson algorithm.\label{Noncentered_algorithm}}
\end{algorithm}

\section{Numerical results}\label{sec:Application}
In this section, we present examples of Bayesian inversion using a multi-layer Gaussian prior presented in the previous section. 
Our main focus in this section is to show the effectiveness of the proposed finite-dimensional approximation method for selected examples. Therefore, we will not discuss the properties of the MCMC algorithm used to generate the samples as they are based on the MCMC algorithms described in \cite{Cui2016,Dunlop2018}. We aim at acceptance ratio between 25-50 \%, which is obtained by tuning the pCN step size $s$ in Algorithm \ref{Noncentered_algorithm}.  %
Our experience in the numerical implementations below indicates that the MCMC algorithm based on the pCN and non-centered algorithm is quite robust. For one and two hyperprior layers implementation, the step size $s$ does not need to be extremely small. The step sizes $s$ in the first and second examples are varying around $10^{-1}$ to $10^{-3}$.
	
	\subsection{Continuous-time random processes with finite-time discrete measurements}\label{ssec:OneD}
	In this section, we consider the application of the proposed technique to address the non-parametric denoising of two piecewise smooth signals. The first test signal is a rectangular shape signal  where the value is zero except on interval $[0.2,0.8]$. The second test signal is a combination between a smooth bell shaped signal and a rectangular signal \cite{Roininen2016}:
	\begin{align}
		\upsilon_{rect}(t) =&
		\begin{cases}
			1 , & t \in [0.2,0.8],\\
			0 , & \text{otherwise}.
		\end{cases}, & \upsilon_{bell,rect}(t) =
		\begin{cases}
			\exp\left(4-\frac{1}{2t - 4t^2}\right) , & t \in (0,0.5),\\
			1 , & t \in [0.7,0.8],\\
			-1 , & t \in (0.8,0.9],\\
			0 , & \text{otherwise}.
		\end{cases} 
	\end{align}

	Previously, in \cite{Emzir2019}, for $J=1$ and with the unknown signal  $\upsilon_{bell,rect}(t) $, we have demonstrated that upon increasing the number of the Fourier basis functions, the $L^2$ error between the ground truth and the posterior sample mean decreased significantly.  

	In this section, we will compare the estimation results using one hyperprior and two hyperprior layers respectively. To allow high variation near points of discontinuities, the length-scale of $\upsilon$ is expected to be smaller around the discontinuities than the rest of the domain. 	We take measurement of on one dimensional grid of $2^{8}$ points and set the standard deviation of the measurement noise to be $0.1$. The proposed algorithm is tested with $N = 2^6 -1$. After estimation, we reconstruct the signal using inverse Fourier transform with a finer grid with $2^{8}$ points equally spaced between zero and one. We take ten million samples for each MCMC run.

	To have a fair comparison, we use the same measurement record for each run with different $J$. Figures \ref{figs:SimulationResult1D_box} and \ref{figs:SimulationResult1D} show the reconstructed signals and their respective length-scale estimations. It can be clearly seen that the addition of another hyper prior layer improves the reconstruction result for the unknown signals. For $J =2$,  although there are no sudden drops near the points of discontinuities, the length-scale value is sufficiently high in the a smooth part of $\upsilon$ signal, and substantially low near the points of discontinuities. This variation translates to a better smoothness detection, as can be seen in Figures \ref{fig:vt_3Layer_box}, \ref{fig:vt_3Layer},  \ref{fig:u_1t_3Layer_box}, and \ref{fig:u_1t_3Layer}. In contrast, Figures \ref{fig:vt_2Layer_box} and \ref{fig:vt_2Layer} show that when $J=1$, the posterior sample means are considerably overfitting the data on the smooth part of the signals. Although the length-scale drops suddenly near the points of discontinuities of the ground truth $\upsilon$, the variation of $\ell$ is limited (see \Fref{fig:u_0t_2Layer_box} and \ref{fig:u_0t_2Layer}). This contributes to a decreased smoothness in the smooth region of the sample mean. The quantitative performance is given in Table \ref{tbl:MSE_PSNR_1D}.

	\newcommand\BoxNLThreeFolder{result-25-Nov-2019_19_37}
	\newcommand\SPDEBoxNLThreeNFourier{64}
	\newcommand\SPDEBoxNLThreeNExtended{256}
	\newcommand\SPDEBoxNLThreeNSamples{10000000}
	\FPeval{SPDEBoxNLThreeMSE}{round(0.003319353651622125:3)}
	\FPeval{SPDEBoxNLThreeLTwo}{round(0.9218213139298006:3)}
	\FPeval{PSNRSPDEBoxNLThree}{round(25.605341346780193:3)}
	\FPeval{SPDEBoxNLThreeLTwoFreq}{round(0.029549868808016667:3)}
	\FPeval{PSNRSPDEBoxNLThreeFreq}{round(49.466565200896284:3)}

	\newcommand\BellNLThreeFolder{result-26-Nov-2019_04_37}
	\newcommand\SPDEBellNLThreeNFourier{64}
	\newcommand\SPDEBellNLThreeNExtended{256}
	\newcommand\SPDEBellNLThreeNSamples{10000000}
	\FPeval{SPDEBellNLThreeMSE}{round(0.008501572872881807:3)}
	\FPeval{SPDEBellNLThreeLTwo}{round(1.475263588467411:3)}
	\FPeval{PSNRSPDEBellNLThree}{round(21.80087287178477:3)}
	\FPeval{SPDEBellNLThreeLTwoFreq}{round(0.03356259772835891:3)}
	\FPeval{PSNRSPDEBellNLThreeFreq}{round(48.640554077045095:3)}

	\newcommand\BoxNLTwoFolder{result-26-Nov-2019_17_10}
	\newcommand\SPDEBoxNLTwoNFourier{64}
	\newcommand\SPDEBoxNLTwoNExtended{256}
	\newcommand\SPDEBoxNLTwoNSamples{10000000}
	\FPeval{SPDEBoxNLTwoMSE}{round(0.00425501130865325:3)}
	\FPeval{SPDEBoxNLTwoLTwo}{round(1.0436871633852896:3)}
	\FPeval{PSNRSPDEBoxNLTwo}{round(24.3253575853686:3)}
	\FPeval{SPDEBoxNLTwoLTwoFreq}{round(0.03716689405477204:3)}
	\FPeval{PSNRSPDEBoxNLTwoFreq}{round(47.27303911781477:3)}

	\newcommand\BellNLTwoFolder{result-26-Nov-2019_16_02}
	\newcommand\SPDEBellNLTwoNFourier{64}
	\newcommand\SPDEBellNLTwoNExtended{256}
	\newcommand\SPDEBellNLTwoNSamples{10000000}
	\FPeval{SPDEBellNLTwoMSE}{round(0.009108251289068831:3)}
	\FPeval{SPDEBellNLTwoLTwo}{round(1.5269945415755815:3)}
	\FPeval{PSNRSPDEBellNLTwo}{round(21.603239507383538:3)}
	\FPeval{SPDEBellNLTwoMSEFreq}{round(0.003319353651622125:3)}
	\FPeval{SPDEBellNLTwoLTwoFreq}{round(0.04034508163432716:3)}
	\FPeval{PSNRSPDEBellNLTwoFreq}{round(47.1435773154203:3)}

	\begin{table}[]
		\centering
		\begin{tabular}{|c|l|l|l|l|}
		\hline
		\multicolumn{1}{|l|}{} & \multicolumn{2}{c|}{Rectangle} & \multicolumn{2}{c|}{Bell-Rectangle} \\ \hline
		 $J $ & \multicolumn{1}{c|}{$L^2$ error} & \multicolumn{1}{c|}{PSNR} & \multicolumn{1}{c|}{$L^2$ error} & \multicolumn{1}{c|}{PSNR} \\ \hline
		1 & \SPDEBoxNLTwoLTwo & \PSNRSPDEBoxNLTwo & \SPDEBellNLTwoLTwo & \PSNRSPDEBellNLTwo  \\ \hline
		2 & \SPDEBoxNLThreeLTwo & \PSNRSPDEBoxNLThree & \SPDEBellNLThreeLTwo & \PSNRSPDEBellNLThree  \\ \hline
		\end{tabular} 
		\caption{Quantitative performance comparison of a shallow Gaussian fields prior inversion for one dimensional signal in Section \ref{ssec:OneD}. \label{tbl:MSE_PSNR_1D}}
		\end{table}

	\begin{figure}
		\centering
		\subfloat[]{\label{fig:vt_3Layer_box}\includegraphics[width=0.5\linewidth]{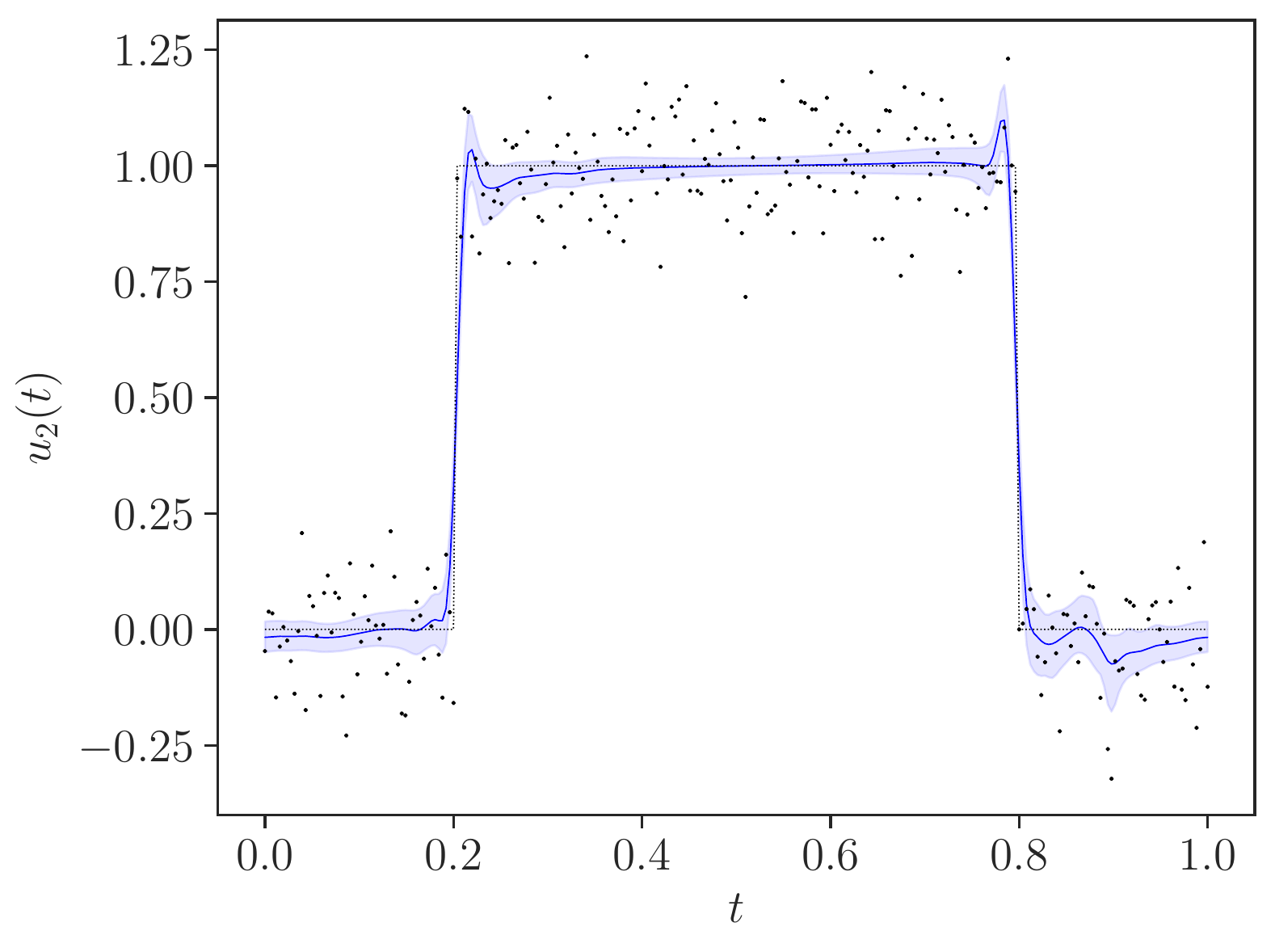}}
		\subfloat[]{\label{fig:vt_2Layer_box}\includegraphics[width=0.5\linewidth]{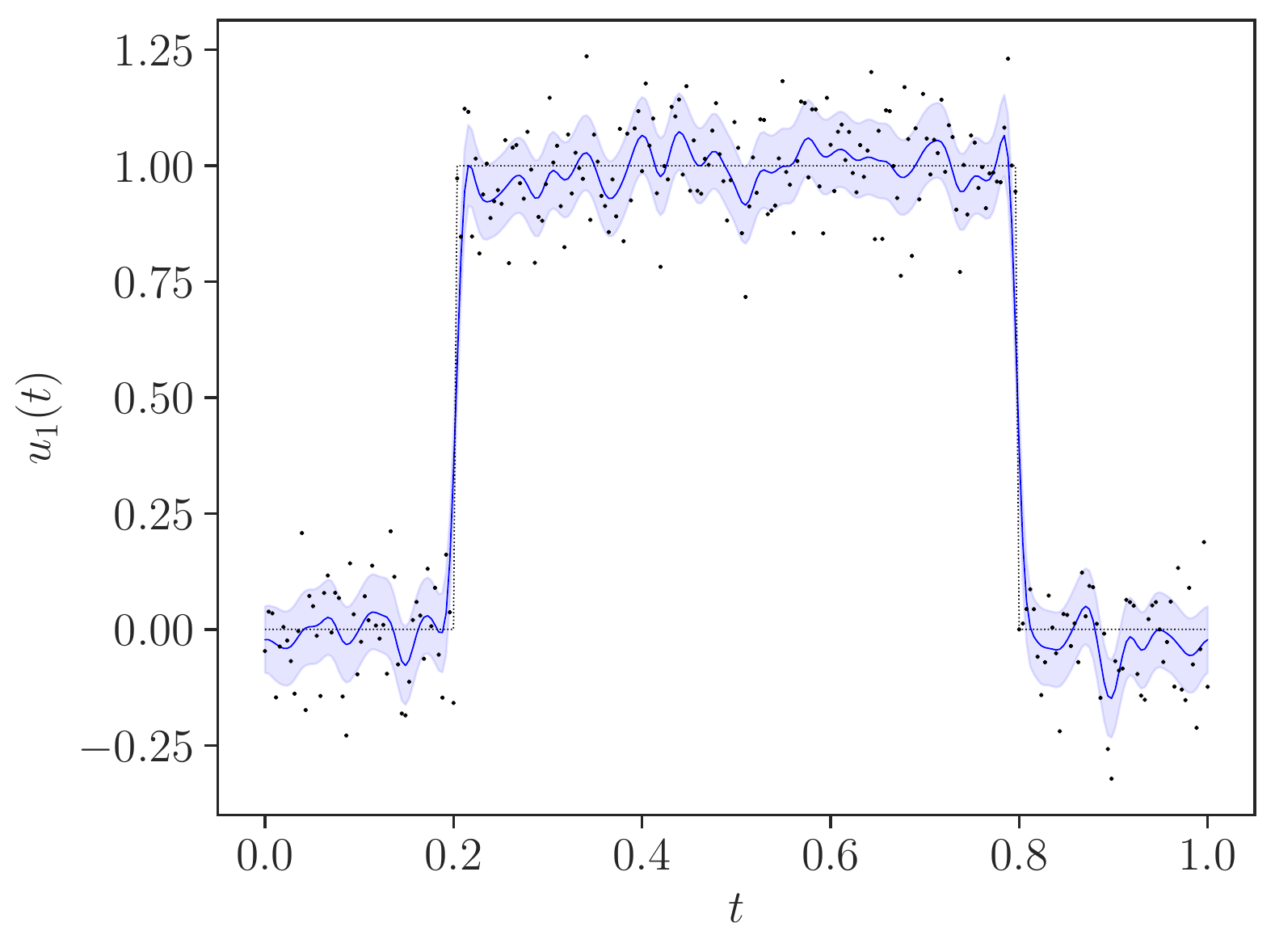}}\\
		\subfloat[]{\label{fig:u_1t_3Layer_box}\includegraphics[width=0.5\linewidth]{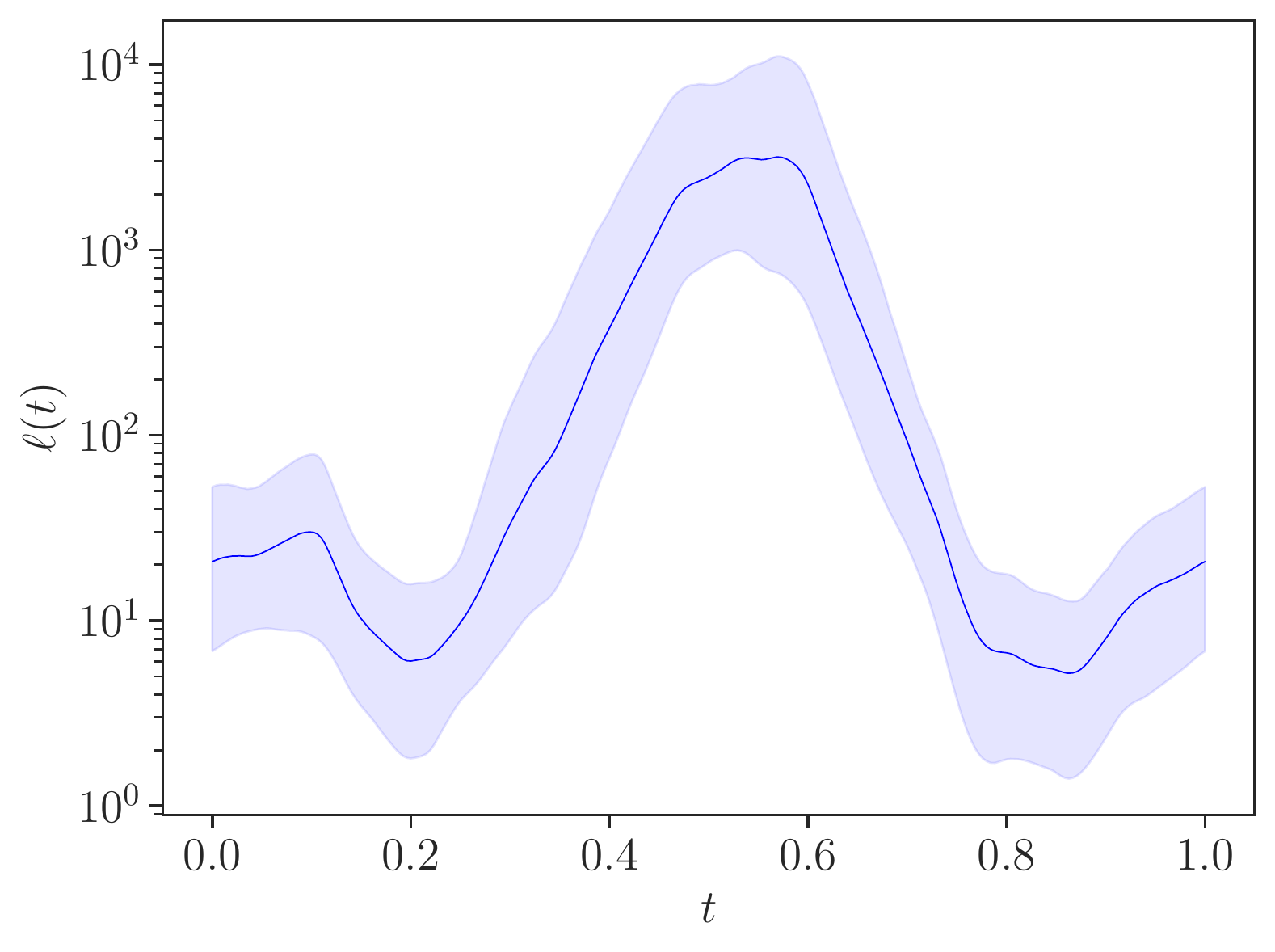}}
		\subfloat[]{\label{fig:u_0t_2Layer_box}\includegraphics[width=0.5\linewidth]{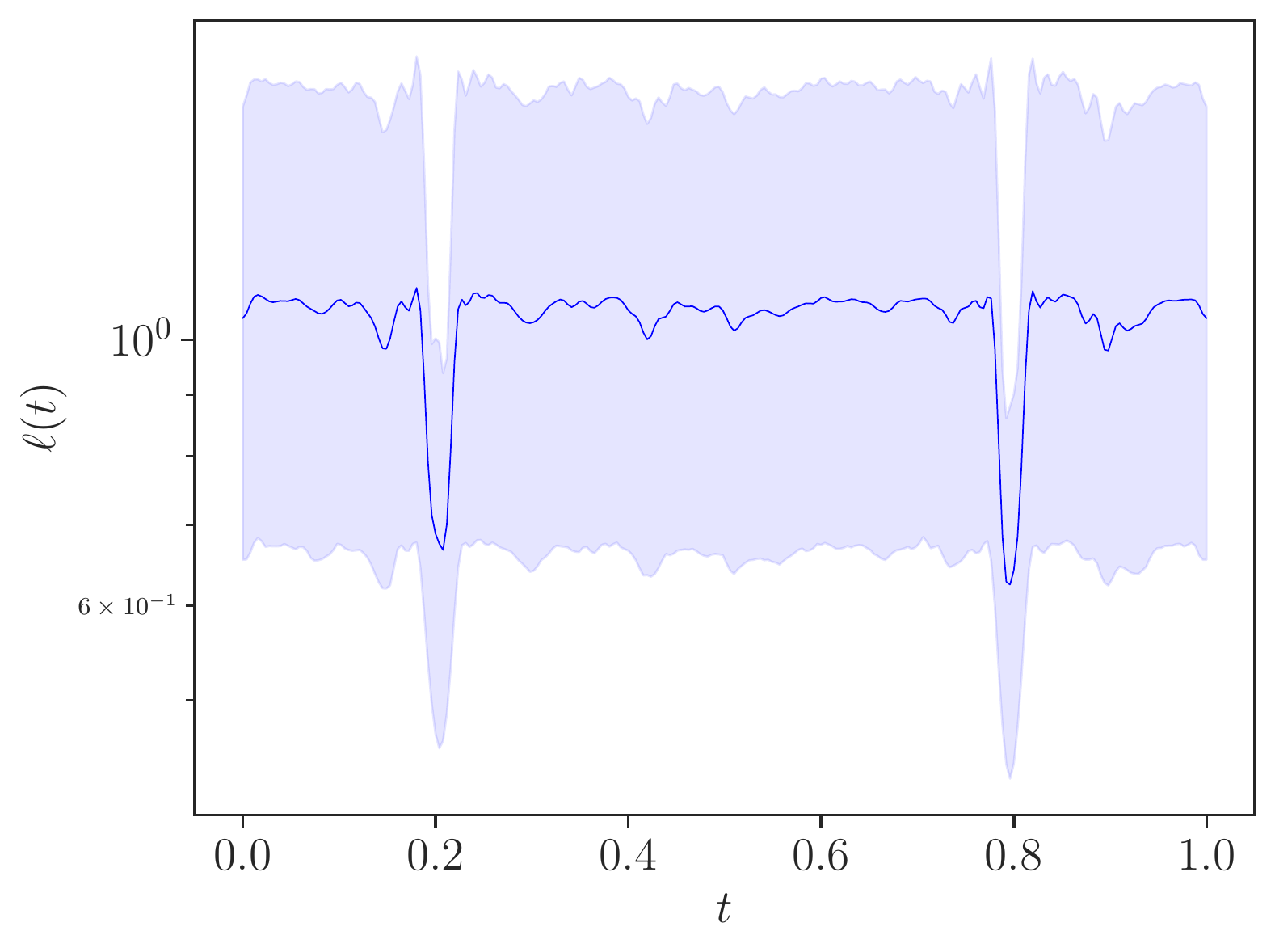}}

		\caption{Simulation results of example in Section \ref{ssec:OneD}. Figures \ref{fig:vt_3Layer_box}, \ref{fig:u_1t_3Layer_box}, \ref{fig:vt_2Layer_box}, and \ref{fig:u_0t_2Layer_box} describe the lowest random fields from $u$ for $J=2$ and $J=1$, and their length-scales. In \Fref{fig:vt_3Layer_box} and \ref{fig:vt_2Layer_box}, the blue and the black lines are the mean Fourier inverse of the samples in the 95 \% confidence shades of the lowest layer and the original unknown signals respectively. The blue line and the shades in the remaining figures are sample means and 95 \% confidence interval of the estimated length-scales.\label{figs:SimulationResult1D_box}}

		\end{figure}

		\begin{figure}
			\centering
			\subfloat[]{\label{fig:vt_3Layer}\includegraphics[width=0.5\linewidth]{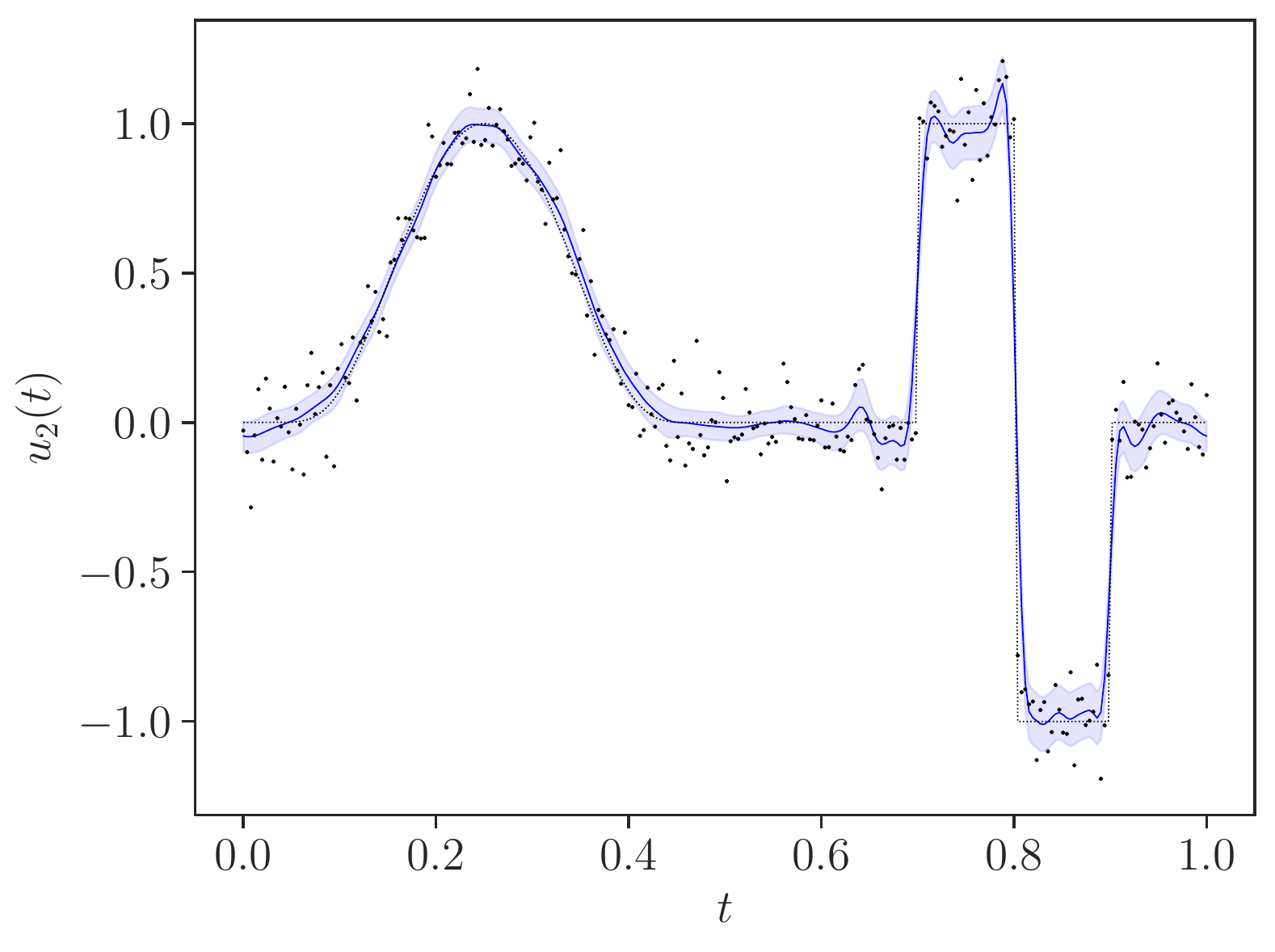}}
			\subfloat[]{\label{fig:vt_2Layer}\includegraphics[width=0.5\linewidth]{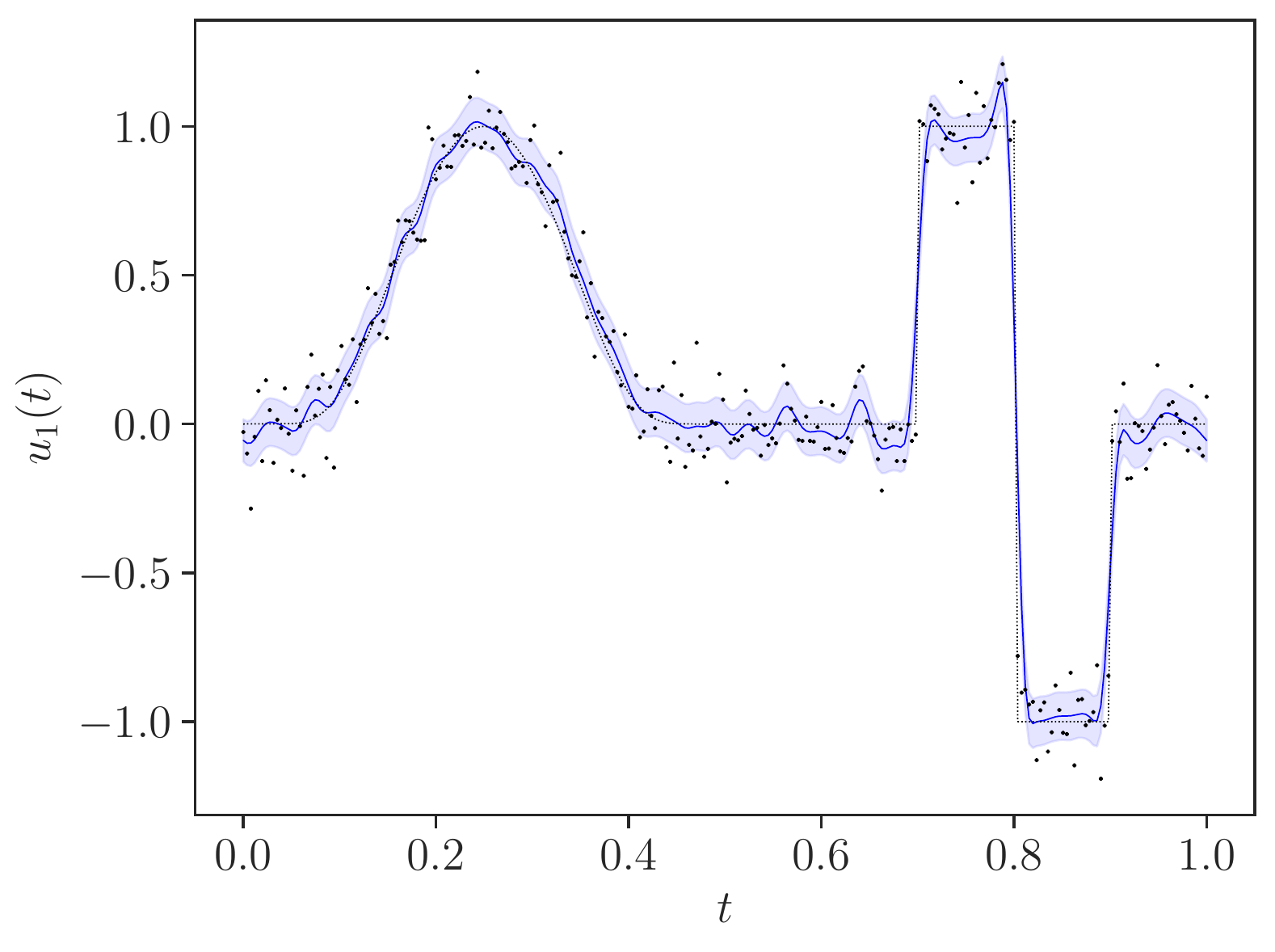}}\\
			\subfloat[]{\label{fig:u_1t_3Layer}\includegraphics[width=0.5\linewidth]{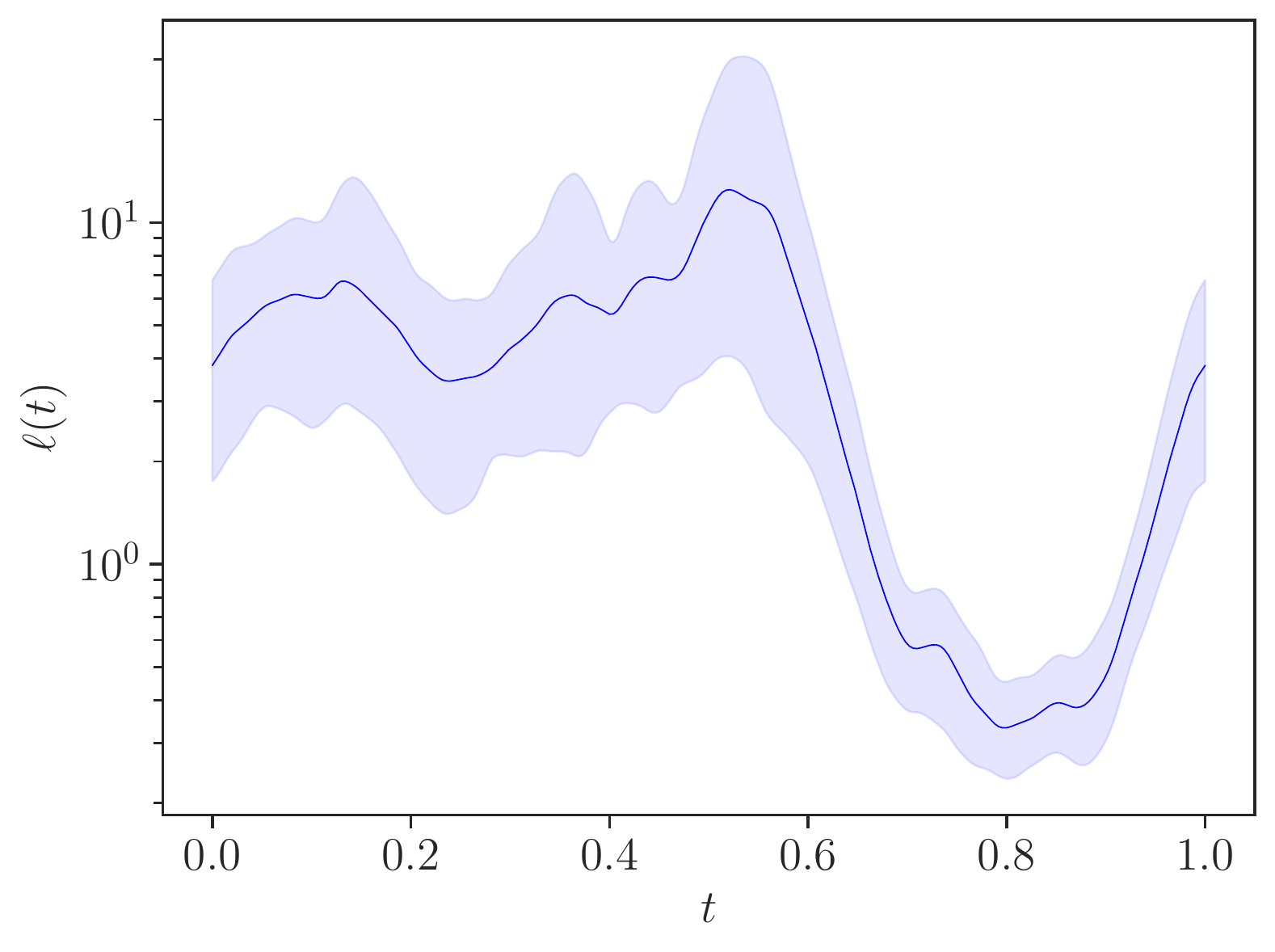}}
			\subfloat[]{\label{fig:u_0t_2Layer}\includegraphics[width=0.5\linewidth]{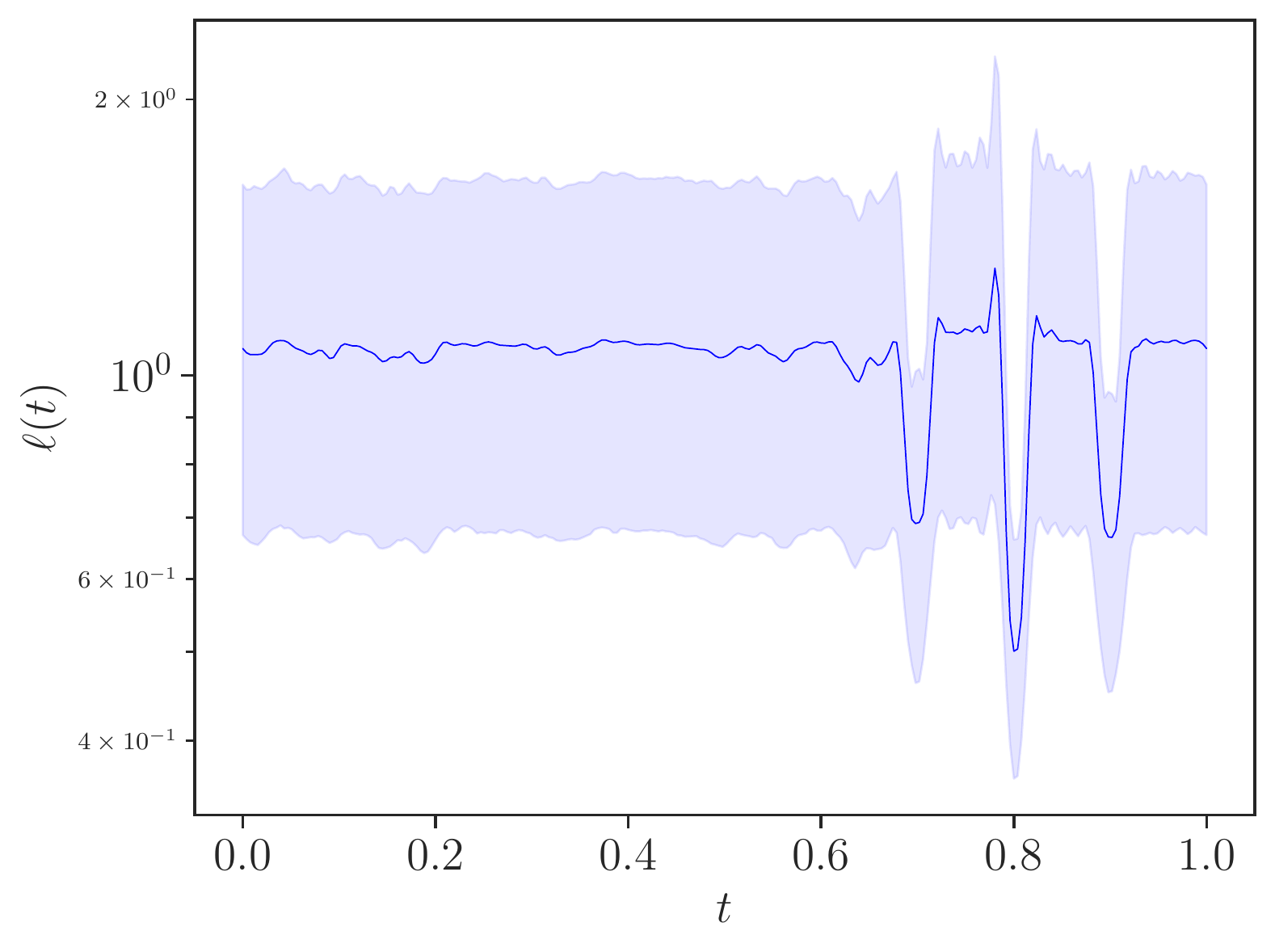}}

			\caption{Similar to \Fref{figs:SimulationResult1D_box}, with the unknown is $v_{bell,rect}$ \label{figs:SimulationResult1D}.}
			\end{figure}

\subsection{X-ray tomography}\label{ssec:Tomography}

In this section, we apply the methods that we have developed to the X-ray tomography reconstruction problem. For the tomography problem, the linear functional is given by a line integration known as the Radon transform \cite{Deans1983,Natterer2001}.%

Assume that the field of interest has support in the a circle with center at $(1/2,1/2)$ and radius equal to half in the two dimensional Euclidean space. Also recall that at a distance $r$ with detection angle $\theta$, we can write the Radon transform of the Fourier basis $\phi_{\mathbf{k}} = \exp\left(i 2 \pi \mathbf{k}^\top \mathbf{x}\right)$ as below
\begin{align*}
	\expval{\phi_{\mathbf{k}},H_{r,\theta}} = \int_{\Omega} \chi(x-\frac{1}{2},y-\frac{1}{2}) \exp\left(i 2 \pi \mathbf{k}^\top \mathbf{x}\right) \delta(r-((x-\frac{1}{2})\cos \theta+(y-\frac{1}{2}) \sin \theta)) d\mathbf{x},
\end{align*}
where $\chi(x,y)$ is an indicator function with support in $(x^2+y^2<\frac{1}{4})$.
Introducing a rotation matrix $\mathbf{R}_{\theta}$, and $\mathbf{p} = [p\;q]^\top$, and $x' = x-\frac{1}{2}$,$y' = y-\frac{1}{2}$, $\mathbf{p} =  	\mathbf{R}_{\theta} \mathbf{x}$
we can write
\begin{align*}
	\expval{\phi_{\mathbf{k}},H_{r,\theta}}
	=& \exp(i \pi (k_x+k_y)) \int_{\Omega} \chi(x',y')  \exp\left(i 2 \pi (\mathbf{R}_{\theta}\mathbf{k})^\top \mathbf{p}\right) \delta(r-p) d\mathbf{p}.
\end{align*}
Using the assumption we have mentioned and $\tilde{\mathbf{k}} =[\tilde{k}_x\;\tilde{k}_y]^\top = \mathbf{R}_{\theta}\mathbf{k}$, we end up with
$
	\expval{\phi_{\mathbf{k}},H_{r,\theta}} 
	 = \exp(i \pi (k_x+k_y)) \exp(i 2\pi \tilde{k}_x r) \frac{1}{\pi \tilde{k}_y}[\sin(2 \pi \tilde{k}_y \sqrt{\frac{1}{4}-r^2} )].
$
Notice when $\tilde{k}_y =0$ we replace the above equation with its limit, that is,
$
	\lim_{\tilde{k}_y \rightarrow 0} \expval{\phi_{\mathbf{k}},H_{r,\theta}} 
	= 2 \sqrt{\frac{1}{4}-r^2} \exp(i \pi (k_x+k_y)) \exp(i 2\pi \tilde{k}_x r).
$
\newcommand\SmallTomographyImage{0.5\linewidth}
\newcommand\PhantomImagePixelLength{511}
\newcommand\DenseTomographySimulationFolder{result-11-Mar-2020_13_35_33}
\newcommand\SPDEDenseNFourier{31}
\newcommand\SPDEDenseNTheta{45}
\newcommand\SPDEDenseNSamples{1000000}
\newcommand\SPDEDenseStepSize{0.018858515}
\FPeval{SPDEDenseLTwo}{round(44.79520552714318:3)}
\FPeval{FBPDenseLTwo}{round(47.15639562030507:3)}
\FPeval{TikhoDenseLTwo}{round(42.145298:3)}
\FPeval{PSNRSPDEDense}{round(22.2461162251573:3)}
\FPeval{PSNRFBPDense}{round(23.398876988985933:3)}
\FPeval{PSNRTikhoDense}{round(21.673435293984053:3)}

\newcommand\SparseTomographySimulationFolder{result-11-Mar-2020_13_46_21}
\newcommand\SPDESparseNFourier{31}
\newcommand\SPDESparseNTheta{18}
\newcommand\SPDESparseNSamples{3000000}
\newcommand\SPDESparseStepSize{0.29353017}
\FPeval{SPDESparseLTwo}{round(124.234237501093:3)}
\FPeval{FBPSparseLTwo}{round(242.01880034542657:3)}
\FPeval{PSNRSPDESparse}{round(18.85348100523729:3)}
\FPeval{PSNRFBPSparse}{round(21.179139070469787:3)}

\begin{table}[!h]
	\centering
	\begin{tabular}{|c|l|l|l|}
	\hline
	  & Shallow-GP & FBP & Tikhonov \\ \hline
	PSNR & \PSNRSPDEDense & \PSNRFBPDense & \PSNRTikhoDense \\ \hline
	$L^2$ Error &\SPDEDenseLTwo & \FBPDenseLTwo & \TikhoDenseLTwo \\ \hline
	\end{tabular} 
	\caption{Quantitative performance comparison of a shallow Gaussian field Bayesian inversion for tomography application in Section \ref{ssec:Tomography}. \label{tbl:MSE_PSNR_Tomography}}
	\end{table}

We modify the MCMC implementation used for the previous one-dimensional example to suit for a GPU architecture. %
For our test comparison in X-ray tomography, the Shepp-Logan phantom with $\PhantomImagePixelLength \times \PhantomImagePixelLength$ resolution is used (see \Fref{fig:target_image_dense_shepp}). We will use this phantom to evaluate the proposed method. 

We take $\SPDEDenseNTheta$ sparsely full projections out of $180$. 
The measurement is corrupted by a white Gaussian noise with standard deviation $0.2$. Due to the restriction in GPU memory, $J $ is set to one. To excel the speed Fourier transform and inverse Fourier transform computations, we use FFT and IFFT routine from CUPY \cite{Okuta2017}. For a performance comparison, we perform the filtered back projection (FBP) on sinogram using iradon routine from Skimage \cite{Walt2014}. A Tikhonov regularization is also used to reconstruct the Fourier coefficients of the unknown field $\mathbf{u}_J$ \cite{Mueller2012}. The Tikhonov regularization parameter $\lambda$ is selected to be $5\times 10^{-2}$ based on the best $L^2$ error and PSNR performance. We set the Fourier basis number $n$ to $\SPDEDenseNFourier$. The total number of parameter for each layer is 1985, which makes the total number of parameters for all layers 3970. The total number of parameters in this example is greatly reduced compared to \cite{suuronen2020enhancing} where each pixel in the target image count as a parameter, that is, for our example it translates to $261121$ parameters.

\Fref{figs:Tomography_dense_Shepp} shows that compared to the FBP and Tikhonov regularization  reconstructions, the posterior sample mean of our MCMC method resulted in an image with less streak artifact and noise. The features of the phantom appear much more clear compared to those on the FBP and Tikhonov reconstructions. As examples, examine the mouth parts, dark circles between eyes and at the forehead, and the two eyes are both relatively much more clear than the other two. Nonetheless, since we set $n$ only $\SPDEDenseNFourier$, the edge of phantom face which has very high values is not fully recovered as much as the FBP reconstruction. The posterior mean produces a lower $L^2$ error ($\SPDEDenseLTwo$) compared to the FBP reconstruction ($\FBPDenseLTwo$), but higher than Tikhonov regularization method, ($\TikhoDenseLTwo$). However, it has a slightly  higher PSNR, $\PSNRSPDEDense$ compared to Tikhonov regularization method, $\PSNRTikhoDense$. Ideally we could double $n$ in our proposed Bayesian method to get a much better reconstruction. However, it is not possible to accomplish this within our current setup due to a restriction on the GPU memory. 
We can fairly conclude that the use of a shallow Gaussian field prior with $n=31$ resulted in a highly reduced amount of artifact at the expense of light blur at the edge. As in the one dimensional example, adding another Gaussian field layer might help to increase the sharpness of the edge in the X-ray tomography application.

\begin{figure}[!h]
	\centering
	\subfloat[]{\label{fig:target_image_dense_shepp}\includegraphics[width=\SmallTomographyImage]{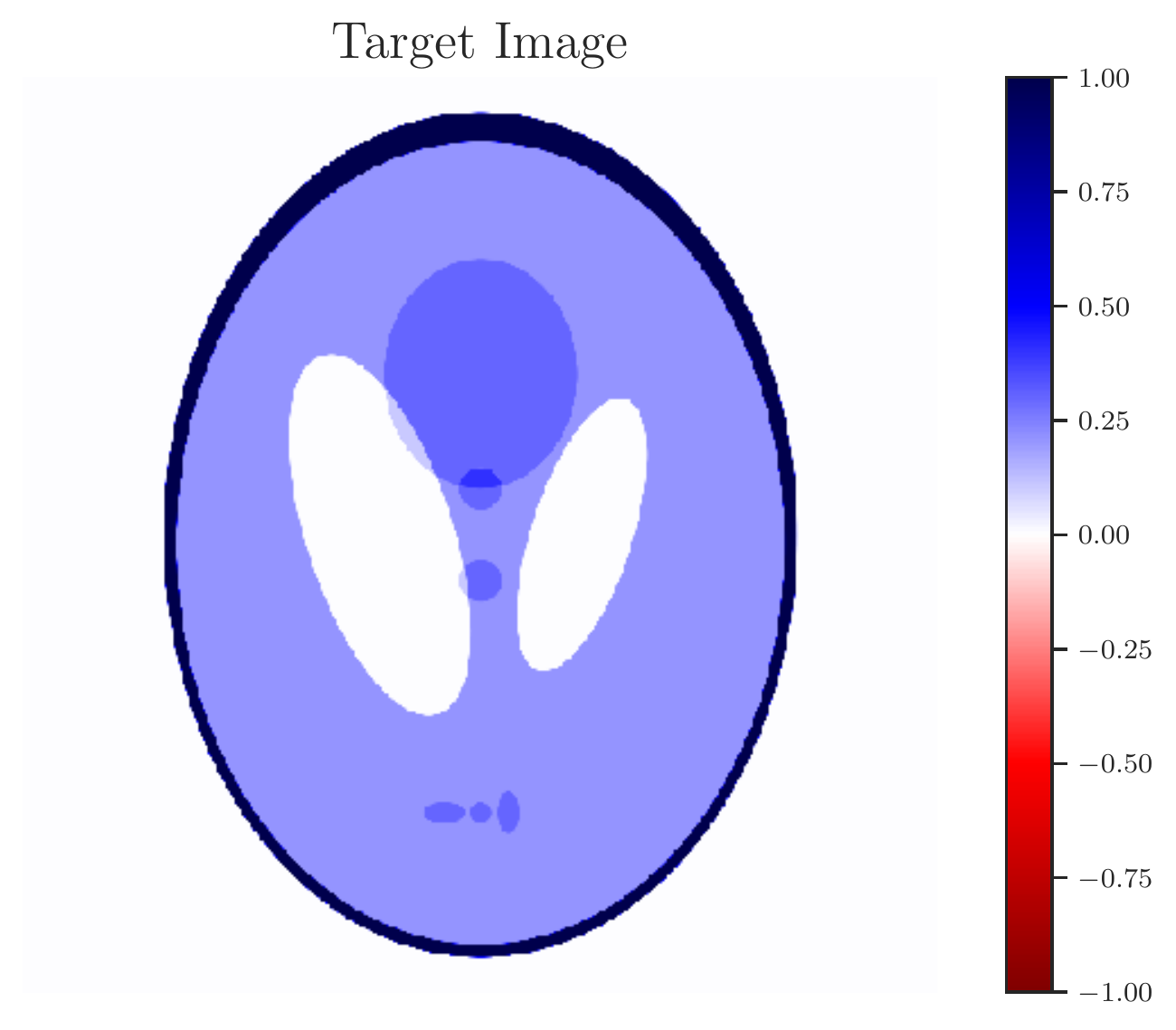}}	
	\subfloat[]{\label{fig:ri_n_dense_shepp}\includegraphics[width=\SmallTomographyImage]{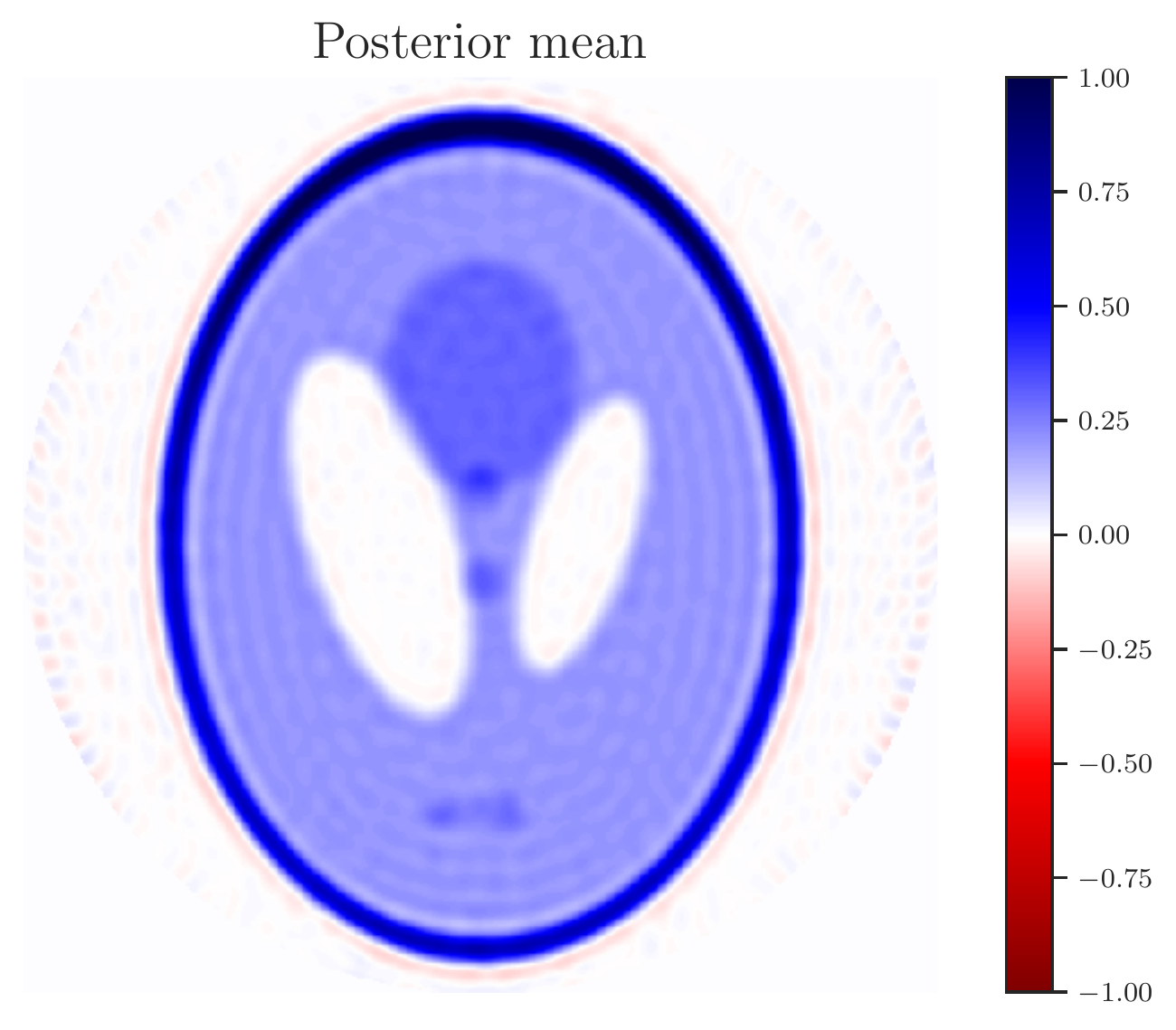}}\\
	\subfloat[]{\label{fig:ri_compare_dense_shepp}\includegraphics[width=\SmallTomographyImage]{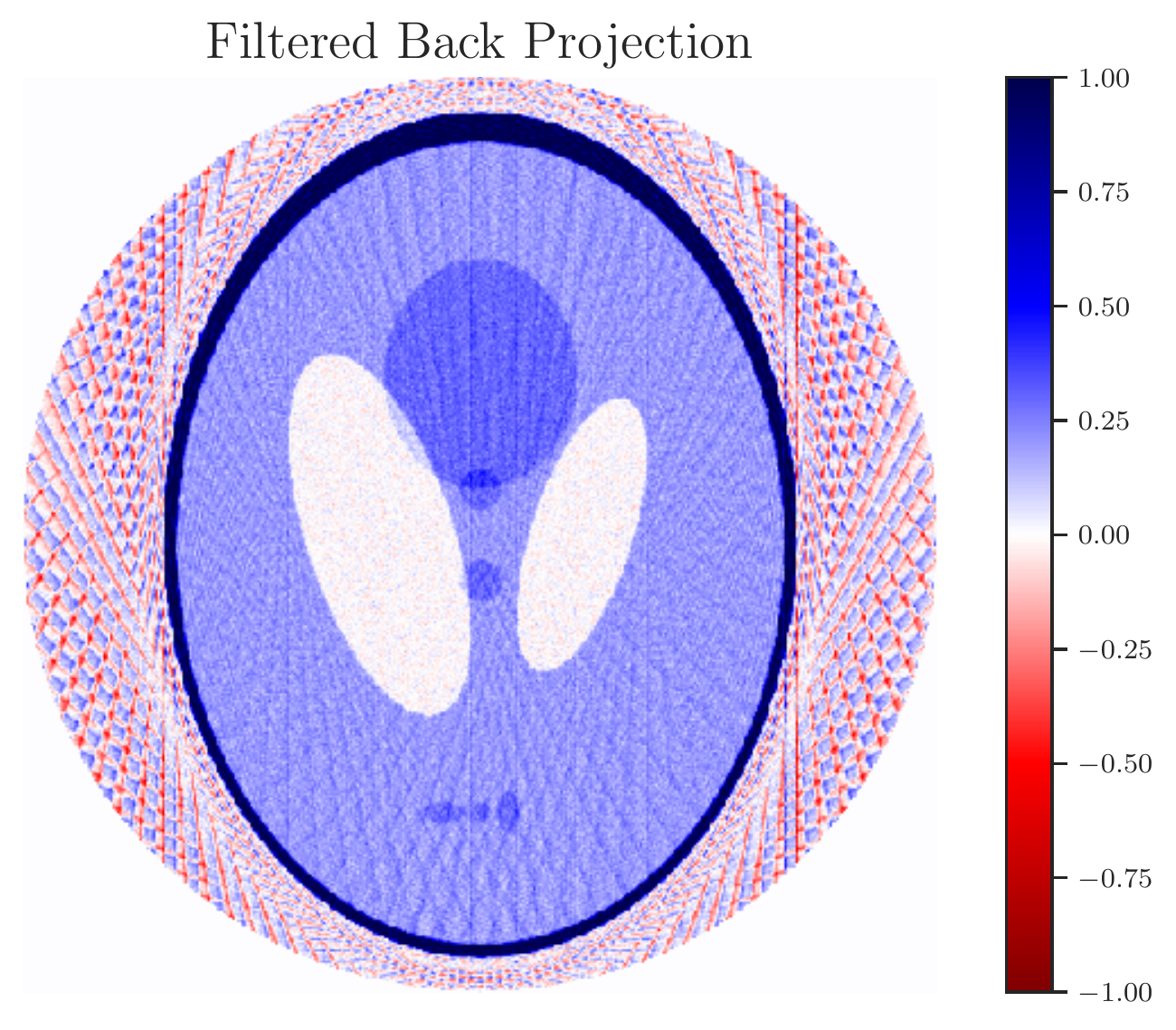}}
	\subfloat[]{\label{fig:ri_Tikhonov_2_dense_shepp}\includegraphics[width=\SmallTomographyImage]{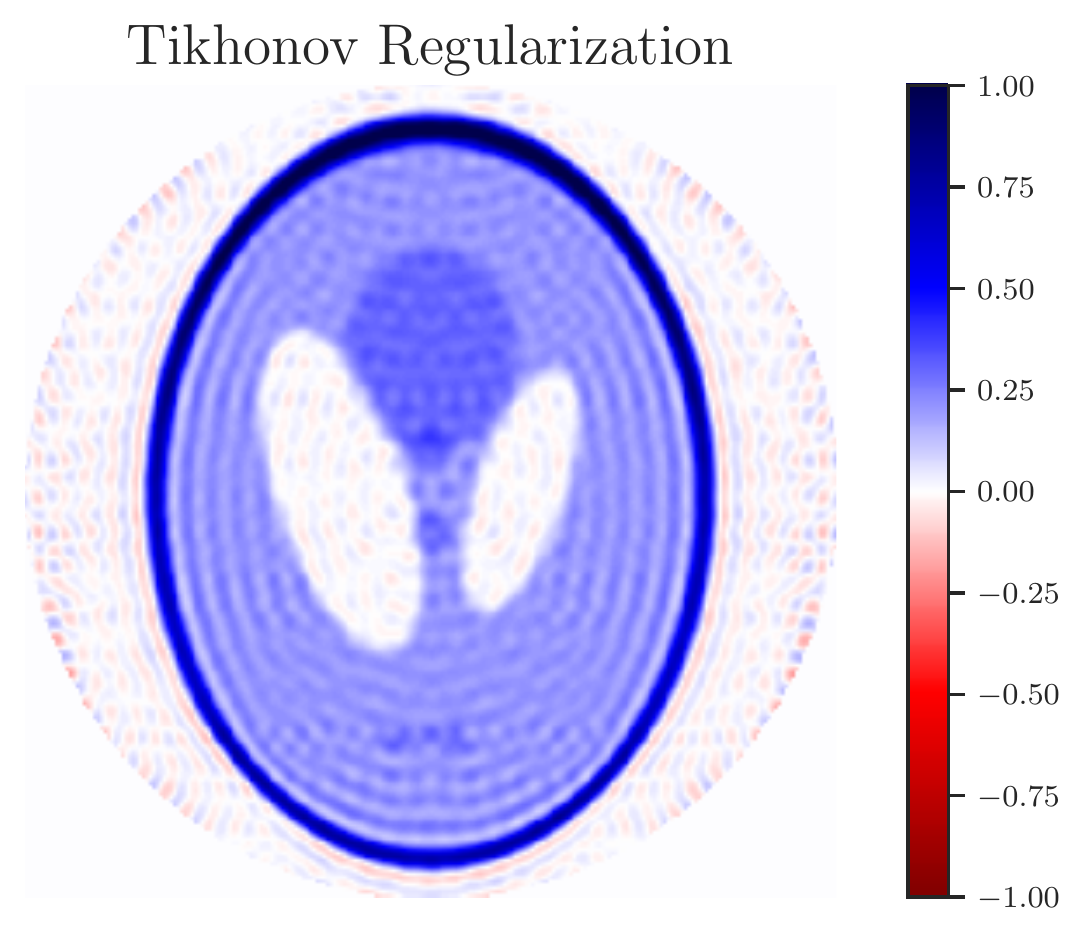}}
	
	\caption{
	Simulation result of a shallow SPDE with $J=1$ and $n = \SPDEDenseNFourier$, where the number of projections is $\SPDEDenseNTheta$. Figures \ref{fig:ri_n_dense_shepp} and \ref{fig:ri_Tikhonov_2_dense_shepp} show posterior means of field $\upsilon$ and the Tikhonov regularization result using $\lambda=5\times 10^{-2}$, while \Fref{fig:ri_compare_dense_shepp} shows the FBP reconstruction. 
	\label{figs:Tomography_dense_Shepp}}
	\end{figure}

\section{Conclusion}\label{sec:conclusions}
We have presented a multi-layered Gaussian-field Bayesian inversion using a Galerkin method. We have also shown that our approach enjoys a nice convergence-in-probability property to the weak solution of the forward model. We also showed that it implies weak convergence of the joint posterior distribution of the Gaussian field. This gives an assurance that our proposed method is well defined and robust upon increasing the number of Fourier basis functions. Using the non-centred version of the preconditioned Crank-Nicolson algorithm, we have shown that for a one-dimensional denoising problem, by using two hyperprior layers we could achieve a smoothing preserving and edge detection of the unknown at the same time. For the X-ray tomography problem, with a single hyper-prior layer and a very small number of Fourier basis ($n = 31$), the posterior sample mean of our proposed approach gives an image with less streak artifact and noise compared to the FBP and Tikhonov regularization reconstructions. Although traces of streak artefact and edge blurring still present, the $L^2$ error and PSNR of our proposed method sit in the middle of those from the FBP and Tikhonov regularization. Furthermore, adding another Gaussian field layer might help to increase the sharpness of the edge in the X-ray tomography application. One of future outlook is to apply the method in real data.

\ack
The authors would like to thank Academy of Finland for financial support.	

\section*{References}
\bibliographystyle{ieeetr}
\bibliography{ReferenceAbbrvBibLatex_UTF8}

\end{document}